\documentclass[11pt]{amsart}
\usepackage[mathscr]{euscript}
\DeclareFontFamily{OT1}{rsfs}{}
\DeclareFontShape{OT1}{rsfs}{n}{it}{<-> rsfs10}{}
\DeclareMathAlphabet{\curly}{OT1}{rsfs}{n}{it}

\makeatletter
\newcommand{\eqnum}{\refstepcounter{equation}\textup{\tagform@{\theequation}}}
\makeatother

\newcommand\beq[1]{\begin{equation}\label{#1}}
\newcommand\eeq{\end{equation}}
\newcommand\beqa{\begin{eqnarray*}}
\newcommand\eeqa{\end{eqnarray*}}

\title[Semiorthogonal decomposition]{Semiorthogonal decompositions 
of stable pair moduli spaces via d-critical flips}
\date{}
\author{Yukinobu Toda}

\usepackage{amscd}
\usepackage{amsmath}
\usepackage{amssymb}
\usepackage{amsthm}
\usepackage{float}
\usepackage[dvips]{graphicx}
\usepackage{xypic}

\usepackage{array}
\usepackage{amscd}
\usepackage[all]{xy}

\DeclareFontFamily{U}{rsfs}{%
\skewchar\font127}
\DeclareFontShape{U}{rsfs}{m}{n}{%
<-6>rsfs5<6-8.5>rsfs7<8.5->rsfs10}{}
\DeclareSymbolFont{rsfs}{U}{rsfs}{m}{n}
\DeclareSymbolFontAlphabet
{\mathrsfs}{rsfs}
\DeclareRobustCommand*\rsfs{%
\@fontswitch\relax\mathrsfs}

\theoremstyle{plain}
\newtheorem{thm}{Theorem}[section]
\newtheorem{prop}[thm]{Proposition}
\newtheorem{lem}[thm]{Lemma}

\newtheorem{defi}[thm]{Definition}
\newtheorem{rmk}[thm]{Remark}
\newtheorem{cor}[thm]{Corollary}

\newtheorem{step}{Step}
\newtheorem{sstep}{Step}

\newtheorem{prop-defi}[thm]{Proposition-Definition}
\newtheorem{thm-defi}[thm]{Theorem-Definition}
\newtheorem{lem-defi}[thm]{Lemma-Definition}

\newtheorem{assum}[thm]{Assumption}

\newtheorem{exam}[thm]{Example}

\newcommand{\sslash}{/\!\!/}

\newcommand{\aA}{\mathcal{A}}
\newcommand{\bB}{\mathcal{B}}
\newcommand{\cC}{\mathcal{C}}
\newcommand{\dD}{\mathcal{D}}
\newcommand{\eE}{\mathcal{E}}
\newcommand{\fF}{\mathcal{F}}
\newcommand{\gG}{\mathcal{G}}
\newcommand{\hH}{\mathcal{H}}

\newcommand{\lL}{\mathcal{L}}
\newcommand{\mM}{\mathcal{M}}

\newcommand{\oO}{\mathcal{O}}
\newcommand{\pP}{\mathcal{P}}
\newcommand{\qQ}{\mathcal{Q}}

\newcommand{\sS}{\mathcal{S}}
\newcommand{\tT}{\mathcal{T}}
\newcommand{\uU}{\mathcal{U}}

\newcommand{\wW}{\mathcal{W}}

\newcommand{\Hom}{\mathop{\rm Hom}\nolimits}

\newcommand{\dotimes}{\stackrel{\textbf{L}}{\otimes}}
\newcommand{\dR}{\mathbf{R}}
\newcommand{\dL}{\mathbf{L}}

\newcommand{\Hilb}{\mathop{\rm Hilb}\nolimits}

\newcommand{\Pic}{\mathop{\rm Pic}\nolimits}

\newcommand{\id}{\textrm{id}}

\newcommand{\ch}{\mathop{\rm ch}\nolimits}

\newcommand{\td}{\mathop{\rm td}\nolimits}
\newcommand{\Ext}{\mathop{\rm Ext}\nolimits}
\newcommand{\Spec}{\mathop{\rm Spec}\nolimits}
\newcommand{\rank}{\mathop{\rm rank}\nolimits}
\newcommand{\Coh}{\mathop{\rm Coh}\nolimits}

\newcommand{\cneq}{\mathrel{\raise.095ex\hbox{:}\mkern-4.2mu=}}
\newcommand{\eqcn}{\mathrel{=\mkern-4.5mu\raise.095ex\hbox{:}}}

\newcommand{\Cok}{\mathop{\rm Cok}\nolimits}

\newcommand{\Aut}{\mathop{\rm Aut}\nolimits}

\newcommand{\Imm}{\mathop{\rm Im}\nolimits}

\makeatletter
 \renewcommand{\theequation}{%
   \thesection.\arabic{equation}}
  \@addtoreset{equation}{section}
\makeatother

\begin{document}

\begin{abstract}
We show the existence of 
semiorthogonal decompositions (SOD) of Pandharipande-Thomas (PT)
stable pair moduli spaces on Calabi-Yau 3-folds 
with irreducible curve classes, assuming
relevant moduli spaces are non-singular. 
The above result is motivated by
categorifications of 
wall-crossing formula of PT
invariants in the derived category, 
and also a d-critical analogue of Bondal-Orlov, Kawamata's 
D/K equivalence conjecture. 

We also give SOD of 
stable pair moduli spaces 
on K3 surfaces, which 
categorifies Kawai-Yoshioka's 
formula proving Katz-Klemm-Vafa formula for PT invariants 
on K3 surfaces with 
irreducible curve classes. 
\end{abstract}

\maketitle

\setcounter{tocdepth}{1}
\tableofcontents

\section{Introduction}
The purpose of this paper is
to give applications of d-critical birational geometry
proposed in~\cite{Toddbir}
to the study of derived categories of coherent sheaves
on moduli spaces of stable objects on 
Calabi-Yau (CY for short) 3-folds. 
The main result is that
Pandharipande-Thomas (PT for short) stable pair moduli spaces~\cite{PT} on 
CY 3-folds with irreducible curve classes
admit certain 
semiorthogonal decompositions (SOD for short), 
assuming relevant moduli spaces are non-singular. 
Our results are motivated by 
categorifications of wall-crossing formula of 
Donaldson-Thomas invariants on CY 3-folds~\cite{JS, K-S}
in the derived category, and 
also a 
d-critical analogue of D/K equivalence conjecture 
by Bondal-Orlov, Kawamata~\cite{B-O2, Ka1}. 
\subsection{SOD of stable pair moduli spaces}
Let $X$ be a smooth projective CY 3-fold 
over $\mathbb{C}$. 
By definition, a \textit{stable pair}
on $X$ consists of data~\cite{PT}
\begin{align}\label{intro:pair}
(F, s), \ s \colon \oO_X \to F
\end{align}
where $F$ is a pure one dimensional coherent sheaf
on $X$ and $s$ is surjective in dimension one. 
For $\beta \in H_2(X, \mathbb{Z})$ and 
$n \in \mathbb{Z}$, we denote by 
\begin{align}\label{intro:moduli}
P_n(X, \beta)
\end{align}
the moduli space of stable pairs (\ref{intro:pair})
such that $[F]=\beta$ and $\chi(F)=n$, 
where $[F]$ is the homology class of the 
fundamental one 
cycle of $F$. 
The moduli space (\ref{intro:moduli}) is a projective scheme with 
a symmetric perfect obstruction theory. 
The integration of its zero dimensional virtual class defines 
the \textit{PT invariant} 
\begin{align*}
P_{n, \beta} \cneq \int_{[P_n(X, \beta)]^{\rm{vir}}}1 \in \mathbb{Z}.
\end{align*}
The study of PT invariants is one of the central 
topics in curve counting theories
on CY 3-folds (see~\cite{MR3221298}). 

Suppose that
$n\ge 0$ and $\beta$ is an irreducible curve class, i.e. 
$\beta$ is not written as $\beta_1+\beta_2$ for 
effective curve classes $\beta_i$. 
Then we have the following diagram
\begin{align}\label{intro:dia:pt}
\xymatrix{
P_n(X, \beta) \ar[rd]_-{\pi^+} & & P_{-n}(X, \beta) \ar[ld]^-{\pi^-} \\
& U_n(X, \beta). &
}
\end{align}
Here $U_n(X, \beta)$ is the moduli space of one dimensional 
Gieseker stable sheaves $F$ on $X$ with $[F]=\beta$
and $\chi(F)=n$.
The maps $\pi^{\pm}$ are defined by 
\begin{align*}
\pi^+(F, s) \cneq F, \ 
\pi^-(F', s') \cneq \eE xt^2_X(F', \oO_X). 
\end{align*}
For a variety $Y$, we denote by $D^b(Y)$ the 
bounded derived category of coherent sheaves on $Y$. 
The following is the main result in this paper: 
\begin{thm}\label{intro:thm:PT}\emph{(Theorem~\ref{thm:PTpair})}
Suppose that $U_n(X, \beta)$ is non-singular. 
Then $P_{\pm n}(X, \beta)$ are also non-singular, and 
we have the following: 

(i) 
The Fourier-Mukai functor
\begin{align*}
\Phi_P  
 \colon D^b(P_{-n}(X, \beta)) \to D^b(P_n(X, \beta))
\end{align*}
with kernel 
the structure sheaf of the fiber product of (\ref{intro:dia:pt})
is fully-faithful. 

(ii) There is a $\pi^{+}$-ample line bundle 
$\oO_P(1)$ on $P_n(X, \beta)$ such that 
if $n\ge 1$, the functor
\begin{align*}
\Upsilon^i_P \colon D^b(U_n(X, \beta)) \to D^b(P_n(X, \beta))
\end{align*}
defined by $\dL\mathrm{\pi^+}^{\ast}(-) \otimes \oO_P(i)$ is
 fully-faithful. 

(iii) We have the SOD
\begin{align}\label{intro:P:SOD}
D^b(P_n(X, \beta))=
\langle \Imm \Upsilon_P^{-n+1}, \ldots, \Imm \Upsilon_P^{0}, \Imm \Phi_P
 \rangle.
\end{align}
\end{thm}

The result of Theorem~\ref{intro:thm:PT}
will be also applied to some non-compact CY 3-folds. 
We apply Theorem~\ref{intro:thm:PT} in the case 
of 
\begin{align*}
X=\mathrm{Tot}_S(K_S), \ H^i(\oO_S)=0, \ i=1, 2
\end{align*}
where $S$ is a smooth projective surface. 
The assumption of Theorem~\ref{intro:thm:PT}
is satisfied when $-K_S \cdot \beta>0$, and we obtain 
the SOD of derived categories of
relative Hilbert schemes of points on
the universal curve over a complete linear system
on $S$
(see Corollary~\ref{cor:locsurface}).

We also apply Theorem~\ref{intro:thm:PT}
in the case of 
\begin{align*}
X=\mathrm{Tot}_C(L_1 \oplus L_2), \ L_i \in \Pic(C), \
L_1 \otimes L_2 \cong \omega_C
\end{align*}
where $C$ is a smooth projective curve. 
For a generic choice of $L_i$, 
the diagram (\ref{intro:dia:pt})
is a classical diagram of symmetric products of
$C$ and their Abel-Jacobi maps. 
Then 
Theorem~\ref{intro:thm:PT} implies 
the SOD of derived categories of coherent sheaves on 
 symmetric products
of $C$ (see Corollary~\ref{thm:sod:sym}):
\begin{align*}
D^b(C^{[n+g-1]})=
\langle \overbrace{D^b(J_C), \ldots, 
D^b(J_C)}^{n}, D^b(C^{[-n+g-1]}) \rangle. 
\end{align*}
Here $n\in \mathbb{Z}_{\ge 0}$, 
$C^{[k]}$ is the $k$-th symmetric product of $C$, 
$g$ is the genus of $C$
and $J_C$ is the Jacobian of $C$. 
The above SOD seems to give a new result on 
the properties of symmetric products of curves and the 
associated Abel-Jacobi maps.

\subsection{Motivations behind Theorem~\ref{intro:thm:PT}}
We have two 
motivations behind the result of 
 Theorem~\ref{intro:thm:PT}. 
The first one is
to give 
a categorification of the following formula 
(see~\cite{PT3, Tsurvey})
\begin{align}\label{intro:wformula}
P_{n, \beta}-P_{-n, \beta}=(-1)^{n-1} n N_{n, \beta}.
\end{align}
Here $N_{n, \beta} \in \mathbb{Z}$ is the 
integration of the virtual class on $U_n(X, \beta)$. 
The identity (\ref{intro:wformula}) 
is the key ingredient to 
show the rationality of the generating series 
of PT invariants
\begin{align*}
P_{\beta}(X)=
\sum_{n \in \mathbb{Z}} P_{n, \beta}q^n
\end{align*}
when $\beta$ is irreducible
(see~\cite{PT3}). 
As observed in~\cite{Tsurvey}, 
the diagram (\ref{intro:dia:pt}) is 
a wall-crossing diagram in $D^b(X)$, 
and (\ref{intro:wformula}) is 
the associated wall-crossing formula.  
Under the assumption of Theorem~\ref{thm:PTpair}, 
the invariants in (\ref{intro:wformula})
are given by 
\begin{align*}
P_{\pm n, \beta}=(-1)^{n+d-1}e(P_{\pm n}(X, \beta)), \ 
N_{n, \beta}=(-1)^d e(U_n(X, \beta))
\end{align*}
where $d$ is the dimension of $U_n(X, \beta)$. 
Therefore the SOD in (\ref{intro:P:SOD})
categorifies the formula (\ref{intro:wformula}), 
as it 
recovers the formula (\ref{intro:wformula})
by taking the Euler characteristics of the Hochschild 
homologies of both sides of (\ref{intro:P:SOD}). 

The second motivation is to give an evidence 
of a d-critical analogue of 
Bondal-Orlov, Kawamata's D/K equivalence conjecture~\cite{B-O2, Ka1}. 
The original D/K equivalence conjecture 
asserts that for a flip of smooth varieties
$Y^+ \dashrightarrow Y^-$
there exists a fully-faithful
functor 
\begin{align*}
D^b(Y^-) \hookrightarrow D^b(Y^+).
\end{align*}
On the other hand, the diagram (\ref{intro:dia:pt}) is an example of a 
\textit{d-critical flip} introduced in~\cite{Toddbir}.
Therefore 
Theorem~\ref{intro:thm:PT} (i) gives an evidence of a d-critical analogue of 
D/K equivalence conjecture. 
We will come back to this point of view 
in 
Subsection~\ref{subsec:intro:dcrit}.

\subsection{Categorification of Kawai-Yoshioka formula}\label{subsec:intro:KY}
We will apply the arguments of Theorem~\ref{intro:thm:PT}
to show the existence of 
SOD on relative Hilbert schemes
of points associated with linear systems
on K3 surfaces. 
Let $S$ be a smooth projective K3 surface 
such that $\Pic(S)$ is generated by 
$\oO_S(H)$ for an ample divisor $H$
with $H^2=2g-2$. 
Let 
\begin{align}
\pi \colon 
\cC \to \lvert H \rvert=\mathbb{P}^g
\end{align}
 be the universal 
curve. 
Below we fix $n\ge 0$, and define 
\begin{align}\label{intro:hilb}
\cC^{[n+g-1]} \to \mathbb{P}^g
\end{align}
to 
be the 
$\pi$-relative Hilbert scheme of $(n+g-1)$-points. 
The moduli space (\ref{intro:hilb}) is known to be
isomorphic to 
the moduli space of PT stable pairs 
$P_n(S, [H])$ on $S$. 

 For each $k\ge 0$, let
$U_k$ be the moduli space of $H$-Gieseker stable 
sheaves $E$ 
on $S$ such that
\begin{align*}
v(E)=(k, H, k+n) \in H^{2\ast}(S, \mathbb{Z})
\end{align*}
where $v(-)$ is the Mukai vector. 
The moduli space $U_k$ is an irreducible holomoprhic symplectic manifold. 
Let $N\ge 0$ be defined to be the largest 
$k\ge 0$ such that $U_k \neq \emptyset$. 
In this situation, we have the following: 
\begin{thm}\emph{(Corollary~\ref{cor:KY})}\label{intro:thm2}
We have the following SOD:
\begin{align*}
D^b(\cC^{[n+g-1]})
=\langle \aA_0, \aA_1, \ldots, \aA_N \rangle
\end{align*}
where each $\aA_k$ has SOD
\begin{align*}
\aA_k=\langle 
\aA_k^{(1)}, \aA_k^{(2)}, \ldots, \aA_k^{(n+2k)} \rangle
\end{align*}
such that 
each $\aA_k^{(i)}$ is equivalent to 
$D^b(U_k)$. 
\end{thm}
The result of Theorem~\ref{intro:thm2} is proved by 
using the zigzag diagram
\begin{align*}
\xymatrix{
\cC^{[n+g-1]}=\pP_0
 \ar[rd]  &  & \pP_1 \ar[ld] \ar[rd] & & \ar[ld] \cdots \ar[rd]
 & & \pP_{N+1}=\emptyset \ar[ld] \\
& U_0 &  & U_1 & & U_N &
}
\end{align*}
constructed by Kawai-Yoshioka~\cite{KY}.
We show that each step of the above diagram 
is described in terms of d-critical simple flip, 
by investigating wall-crossing diagrams 
on a CY 3-fold $X=S \times C$ for an 
elliptic curve $C$. 
Then Theorem~\ref{intro:thm2} is proved 
applying 
the argument of Theorem~\ref{intro:thm:PT} 
to each step of the diagram.  

The SOD in Theorem~\ref{intro:thm2}
is interpreted as a categorification of Kawai-Yoshioka's formula~\cite{KY} 
for PT invariants on K3 surfaces with irreducible 
curve classes, defined by $P_{n, g} \cneq (-1)^{n-1}e(\cC^{[n+g-1]})$.
Indeed the following formula 
is proved in~\cite{KY} 
\begin{align}\label{intro:euler}
e(\cC^{[n+g-1]}) =\sum_{k=0}^{N}
(n+2k)e(U_k). 
\end{align}
The SOD in Theorem~\ref{intro:thm2}
recovers the formula (\ref{intro:euler}) by taking 
the Euler characteristics of Hochschild homologies of both sides of (\ref{intro:thm2}). 
In~\cite{KY}, the 
formula (\ref{intro:euler}) led to the 
Katz-Klemm-Vafa (KKV) formula for PT
invariants with irreducible curve classes
(see Remark~\ref{rmk:KKV}).

\subsection{D-critical analogue of D/K equivalence conjecture}\label{subsec:intro:dcrit}
Here we explain the notion of 
d-critical flips 
for Joyce's d-critical loci~\cite{JoyceD}, 
and an analogue of D/K equivalence conjecture 
mentioned earlier. 
By definition, a \textit{d-critical locus} 
 consists of data
\begin{align*}
(M, s), \ s \in \Gamma(M, \sS_{M}^0)
\end{align*} where 
$M$ is a $\mathbb{C}$-scheme or an analytic space and 
$\sS_M^0$ is a certain sheaf of $\mathbb{C}$-vector spaces
on $M$. 
The section $s$ is called a \textit{d-critical structure} of $M$. 
Roughly speaking
if $M$ admits a d-critical structure $s$, 
this means that $M$ is locally written as a critical 
locus of some function on a smooth space, and the section 
$s$ remembers how $M$ is locally written as a critical locus. 
If $M$ is a truncation of a derived scheme with a $(-1)$-shifted 
symplectic structure~\cite{PTVV}, then 
it has a canonical d-critical structure~\cite{BBBJ}. 

Let $(M^{\pm}, s^{\pm})$
be two d-critical loci and consider a diagram
of morphisms of schemes or analytic spaces
\begin{align}\label{intro:dcrit}
\xymatrix{
M^{+} \ar[rd] & & M^{-} \ar[ld] \\
& U. &
}
\end{align}
The above diagram is called a
\textit{d-critical flip} 
if it satisfies the following: 
for any
$p \in U$, there is a commutative diagram 
\begin{align}\label{intro:dcrit2}
\xymatrix{
Y^{+} \ar[rd] 
\ar@/_5pt/[rdd]_-{w^{+}}  \ar@{.>}[rr]^-{\phi}& & Y^{-} \ar[ld] 
\ar@/^5pt/[ldd]^-{w^-}\\
& Z \ar[d]_-{w} &  \\
& \mathbb{C} &
}
\end{align}
where $ \phi \colon Y^{+} \dashrightarrow Y^{-}$ is a flip 
of smooth varieties (or complex manifolds), 
such that locally near $p\in U$
there exist isomorphisms of 
$M^{\pm}$ and $\{dw^{\pm}=0\}$ as d-critical loci
(see~\cite[Definition~3.7]{Todbir} for details).
A d-critical flip is called \textit{simple}
if $\phi \colon Y^+ \dashrightarrow Y^-$ is a 
simple toric flip~\cite{Rei92}. 

We expect that an analogue of the D/K equivalence 
conjecture 
may hold 
for d-critical 
loci. 
 Namely 
for a d-critical locus $(M, s)$ 
\footnote{Probably we need to assume that 
$(M, s)$ is 
induced by a $(-1)$-shifted symplectic
derived scheme,
equipped 
with some additional data (orientation data, or something more). }
there may exist a certain 
triangulated category $\dD(M, s)$ such that, 
if the diagram (\ref{intro:dcrit}) is a d-critical flip, we have 
a fully faithful functor
\begin{align}\label{intro:shifted:emb}
\dD(M^{-}, s^{-}) \hookrightarrow 
\dD(M^{+}, s^{+}).
\end{align}
The category $\dD(M^{-}, s^{-})$ may be 
constructed as a gluing of 
$\mathbb{Z}/2\mathbb{Z}$-periodic triangulated categories of 
matrix factorizations
defined locally on each d-critical chart, 
though its construction seems to be a hard 
problem at this moment 
(see~\cite[(J)]{Jslide}, \cite[Section~6.1]{MR3728637}). 

For a flip $Y^+ \dashrightarrow Y^-$ 
in the diagram (\ref{intro:dcrit2}), suppose that 
the D/K equivalence conjecture holds, i.e. 
we have a fully-faithful functor 
\begin{align*}
D^b(Y^-) \hookrightarrow D^b(Y^+).
\end{align*}
Then it induces the fully-faithful functor
(see Theorem~\ref{thm:FF})
\begin{align}\label{intro:Y:FF}
D(Y^-, w^-) \hookrightarrow D(Y^+, w^+)
\end{align}
where $D(Y^{\pm}, w^{\pm})$ are the 
derived factorization categories associated with 
pairs $(Y^{\pm}, w^{\pm})$. 
If the desired 
categories $\dD(M^{\pm}, s^{\pm})$ are
gluing of 
$D(Y^{\pm}, w^{\pm})$ defined locally on $U$, 
then we may try to globalize the functor (\ref{intro:Y:FF}) 
to give a fully-faithful functor (\ref{intro:shifted:emb}). 
If this is possible, 
they 
the numerical realization of semiorthogonal complement 
of the embedding (\ref{intro:shifted:emb})
may recover wall-crossing formula of 
DT invariants~\cite{JS, K-S}
\footnote{See~\cite{Eff} for the relation of cyclic homologies of 
the categories of 
matrix factorizations with hypercohomologies of perverse sheaves of vanishing cycles.}.

For a d-critical flip (\ref{intro:dcrit}), 
suppose that $M^{\pm}$ are smooth, 
so in particular $s^{\pm}=0$.
In this case, we can use
usual derived categories of coherent sheaves to ask 
an analogue of the above question. 
Namely for 
a d-critical flip (\ref{intro:dcrit})
with $M^{\pm}$ smooth, 
we can ask whether we have a  
a fully faithful functor
\begin{align*}
D^b(M^-) \hookrightarrow D^b(M^+).
\end{align*}
The results of Theorem~\ref{intro:thm:PT}, Theorem~\ref{intro:thm2}
are proved by establishing 
such a result in the 
case of d-critical simple flips
(see Theorem~\ref{thm:main}). 

\subsection{Relations to other works}
There exist some recent works studying 
wall-crossing behavior of derived categories of moduli spaces of 
stable objects on algebraic surfaces. 
In~\cite{MR3652079}, Ballard showed the existence of SOD 
under wall-crossing of Gieseker moduli spaces of 
stable sheaves on rational surfaces. 
Also Halpern-Leistner~\cite{HalpK3} announces
that, under wall-crossing of Bridgeland moduli spaces of 
stable objects on K3 surfaces, their derived 
categories are equivalent. The results in this paper
are regarded as CY 3-fold version of these works. 
One of the crucial differences is that, although the 
moduli spaces considered in~\cite{MR3652079, HalpK3} are 
birational under wall-crossing, 
the moduli spaces in this paper are not necessary birational 
under wall-crossing. 
For example the moduli spaces $P_{\pm n}(X, \beta)$
in the diagram (\ref{intro:dia:pt})
have different dimensions if $n>0$. 
Instead the fact that they are birational in d-critical 
birational geometry plays an important role for the 
existence of SOD in Theorem~\ref{intro:thm:PT}, Theorem~\ref{intro:thm2}.

\subsection{Outline of the paper}
The outline of this paper is as follows. 
In Section~\ref{sec:review}, we review basics on 
derived factorization categories which we will use 
in later sections. 
In Section~\ref{sec:sod}, we show 
the existence of SOD of gauged LG models 
on simple flips over a complete local base, and 
describe the relevant kernel objects. 
In Section~\ref{sec:dcrit}, we globalize the result in 
Section~\ref{sec:sod} and show the SOD for formal 
d-critical simple flips. 
In Section~\ref{sec:SODpair}, we use the result in Section~\ref{sec:dcrit}
to show Theorem~\ref{intro:thm:PT}. 
In Section~\ref{sec:KY}, we prove Theorem~\ref{intro:thm2}. 

\subsection{Acknowledgements}
The author is grateful to Yuki Hirano, 
Daniel Halpern-Leistner
and Dominic Joyce for valuable discussions.
The author is supported by World Premier International Research Center
Initiative (WPI initiative), MEXT, Japan, and Grant-in Aid for Scientific
Research grant (No. 26287002) from MEXT, Japan.

\subsection{Notation and Convention}
In this paper, all the varieties and schemes
are defined over $\mathbb{C}$. 
For $\mathbb{C}$-schemes $U$, $S$, $T$
and a morphism $f \colon S \to T$, we set 
\begin{align*}
S_U \cneq S \times U, \ 
f_U \cneq f \times \id_U \colon S_U \to T_U.
\end{align*}

For a variety $Y$, we denote by $D^b(Y)$ 
the bounded derived category of coherent sheaves on $Y$. 
For smooth varieties $Y_1$, $Y_2$ with 
projective morphisms 
$Y_i \to T$, 
and an object $\pP \in D^b(Y_1 \times Y_2)$
supported on $Y_1 \times_T Y_2$, we denote by 
$\Phi^{\pP}$ the \textit{Fourier-Mukai 
functor}
\begin{align*}
\Phi^{\pP}(-) 
\cneq \dR p_{2\ast}(p_1^{\ast}(-) \dotimes \pP) \colon 
D^b(Y_1) \to D^b(Y_2).
\end{align*}
Here $p_i \colon Y_1 \times Y_2 \to Y_i$ are the projections. 
The object $\pP$ is called a \textit{kernel} of the 
functor $\Phi^{\pP}$. 

Recall that a \textit{semiorthogonal decomposition} of a
triangulated category $\dD$ is a collection 
$\cC_1, \ldots, \cC_n$ of full 
triangulated subcategories such that 
$\Hom(\cC_i, \cC_j)=0$ for all $i>j$
and the smallest triangulated 
subcategory of $\dD$ containing 
$\cC_1, \ldots, \cC_n$
coincides with $\dD$. In this case, 
we write $\dD=\langle \cC_1, \ldots, \cC_n \rangle$. 
If each $\cC_i$ is equivalent to 
$D^b(M_i)$ for a variety $M_i$, 
we also write 
$\dD=\langle D^b(M_1), \ldots, D^b(M_n) \rangle$
for simplicity. 

\section{Review of derived factorization categories}\label{sec:review}
In this section, we recall the notion of 
gauged Landau-Ginzburg (LG) models, and the 
associated derived factorization categories
introduced by Positselski. 
For details, we refer to the articles~\cite{MR3366002,MR3366002}
for basics on these notions. 
\subsection{Definitions of derived factorization categories}
Let us consider data (called \textit{gauged LG model})
\begin{align}\label{gauge}
(Y, G, \chi, w)
\end{align}
where $Y$ is a $\mathbb{C}$-scheme, 
$G$ is a reductive algebraic group which acts on $Y$, 
$\chi \colon G \to \mathbb{C}^{\ast}$ is a character 
and $w \in \Gamma(\oO_Y)$ satisfies
 $g^{\ast}w=\chi(g) w$ for any $g \in G$. 
 Given data as above, the \textit{derived factorization category}
 \begin{align}\label{der:fact}
 D_{G}(Y, \chi, w)
 \end{align}
 is defined as a triangulated category, 
whose objects consist of 
\textit{factorizations of }$w$, i.e. 
sequences of $G$-equivariant morphisms of 
$G$-equivariant coherent sheaves $\fF_0$, $\fF_1$ on $Y$
\begin{align}\label{factorization}
\fF_0 \stackrel{\alpha}{\to} \fF_1 \stackrel{\beta}{\to} \fF_0(\chi)
\end{align}
satisfying the following: 
\begin{align*}
\alpha \circ \beta=\cdot w, \ \beta \circ \alpha=\cdot w.
\end{align*}
The category (\ref{der:fact}) is defined to be
 the localization 
of the homotopy category of 
the factorizations (\ref{factorization})
by its subcategory of acyclic factorizations. 
When $Y$ is an affine scheme and
$G=\{1\}$, 
then the category (\ref{der:fact}) is equivalent to the 
triangulated category of matrix factorizations
of $w$ (see~\cite{Orsin}). 
In the case of 
$G=\mathbb{C}^{\ast}$ and $\chi=\id$,
we simply write
\begin{align*}
D_{\mathbb{C}^{\ast}}(Y, w) \cneq D_{\mathbb{C}^{\ast}}(Y, \chi=\id, w).
\end{align*}

For a character $\chi \colon G \to \mathbb{C}^{\ast}$, let 
$\widetilde{\chi}$ be defined by
\begin{align*}
\widetilde{\chi} \colon G \times \mathbb{C}^{\ast} \to \mathbb{C}^{\ast}, \ 
(g, t) \mapsto \chi(g)t.
\end{align*}
We have the functor
\begin{align}\label{equiv:w=00}
\Xi \colon
D_G^b(Y) \to D_{G\times \mathbb{C}^{\ast}}(Y, \widetilde{\chi}, w=0)
\end{align}
where $\mathbb{C}^{\ast}$ acts on $Y$ trivially, 
sending $(\fF^{\bullet}, d) \in D_G^b(Y)$ to 
\begin{align*}
\left(\bigoplus_{i\in \mathbb{Z}}
\fF^{2i}(-i\widetilde{\chi}) \right)
\stackrel{d}{\to}
\left(\bigoplus_{i\in \mathbb{Z}}
\fF^{2i+1}(-i\widetilde{\chi}) \right)
\stackrel{d}{\to}
\left(\bigoplus_{i\in \mathbb{Z}}
\fF^{2i}(-i\widetilde{\chi}) \right)(\widetilde{\chi}).
\end{align*}
When $G=\{1\}$, the functor (\ref{equiv:w=00})
gives the equivalence 
(see~\cite{MR3071664, MR2982435, MR3581302})
\begin{align}\label{equiv:w=0}
\Xi \colon
D^b(Y) \stackrel{\sim}{\to} D_{\mathbb{C}^{\ast}}(Y, 0).
\end{align}

\subsection{Derived functors between derived factorization categories}
Let $(Y, G, \chi, w)$ be a gauged LG model (\ref{gauge}), 
and $W$ be another variety with a $G$-action. 
For a $G$-equivariant projective morphism 
$f \colon W \to Y$, we have another gauged LG model
\begin{align*}
(W, G, \chi, f^{\ast}w).
\end{align*}
Similarly to the usual derived functors between 
derived categories,
if $Y$ is smooth 
we have derived functors
\begin{align*}
\dR f_{\ast} \colon D_G(W, \chi, f^{\ast}w) \to D_G(Y, \chi, w), \\
\dL f^{\ast} \colon D_G(Y, \chi, w) \to D_G(W, \chi, f^{\ast}w).
\end{align*}
Also for another object 
$\pP \in D_G(Y, \chi, w')$, 
we have the derived tensor product
\begin{align*}
\dotimes \pP \colon D_G(Y, \chi, w) \to D_G(Y, \chi, w+w').
\end{align*}
Below we omit the subscripts 
$\dR$, $\dL$ when the relevant 
functors are exact functors of coherent sheaves, e.g. 
write $\dL f^{\ast}$ as $f^{\ast}$ when $f$ is flat. 

Let $Y_1$, $Y_2$ be 
regular $\mathbb{C}$-schemes with  
$G$-actions. 
Let $T$ be a $\mathbb{C}$-scheme 
with a $G$-action and 
consider $G$-equivariant 
projective morphisms
 $Y_i \to T$. 
 Let us take $w \in \Gamma(\oO_T)$
and a character $\chi \colon G \to \mathbb{C}^{\ast}$
 satisfying $g^{\ast}w=\chi(g) w$ for
any $g \in G$. 
We consider the commutative diagram
 \begin{align*}
 \xymatrix{
Y_1 \ar[rd]
\ar@/_5pt/[rdd]_-{w_1}  & & Y_2
\ar[ld]
\ar@/^5pt/[ldd]^-{w_2}\\
& T \ar[d]_-{w} &  \\
& \mathbb{A}^1 &
}
   \end{align*}
   Let $p_i \colon Y_1 \times Y_2 \to Y_i$ be the 
   projection, and $G$ acts on 
$Y_1 \times Y_2$ diagonally. 
   For any object
 \begin{align*}
 \pP \in D_G(Y_1 \times Y_2, \chi, -p_1^{\ast}w_1+p_2^{\ast}w_2)
 \end{align*}
 we have the Fourier-Mukai type functor
 \begin{align*}
 \Psi^{\pP} \cneq \dR p_{2\ast} (p_1^{\ast}(-) \dotimes \pP)
  \colon 
  D_G(Y_1, \chi, w_1) \to D_G(Y_2, \chi, w_2).
  \end{align*}
 Let $i \colon Y_1 \times_T Y_2 \hookrightarrow Y_1 \times Y_2$ be 
the closed embedding. We have the following diagram
\begin{align*}
\xymatrix{
D_G^b(Y_1 \times_T Y_2)  \ar[d]^-{\mathrm{forg}^G}
 \ar[r]^-{\Xi} &
 D_{G \times \mathbb{C}^{\ast}}(Y_1 \times_T Y_2, \widetilde{\chi}, 0) 
\ar[d]^-{\mathrm{forg}^{G}}  \ar[r]^-{\mathrm{forg}^{\mathbb{C}^{\ast}}} & 
D_{G}(Y_1 \times_T Y_2, \chi, 0) \ar[d]^-{i_{\ast}}
 \\
D^b(Y_1 \times_T Y_2) \ar[r]^-{\Xi} &
D_{\mathbb{C}^{\ast}}(Y_1 \times_T Y_2, 0) 
 &  D_{G}(Y_1 \times Y_2, \chi, -p_1^{\ast}w_1+p_2^{\ast}w_2).  
}
\end{align*}
Here $\Xi$ is given in (\ref{equiv:w=00}) and 
$\mathrm{forg}^G$, $\mathrm{forg}^{\mathbb{C}^{\ast}}$
are  
forgetting the $G$-action, $\mathbb{C}^{\ast}$-action respectively. 
For $\pP \in D^b(Y_1 \times_T Y_2)$
and $\Xi(\pP) \in D_{\mathbb{C}^{\ast}}(Y_1 \times_T Y_2, 0)$, 
the following diagram 
commutes: 
\begin{align}\label{commute:Xi}
\xymatrix{
D^b(Y_1) 
\ar[r]_-{\Xi}^-{\sim} \ar[d]_-{\Phi^{\pP}} & D_{\mathbb{C}^{\ast}}(Y_1, 0) 
\ar[d]^-{\Psi^{\Xi(\pP)}} \\
D^b(Y_2) \ar[r]_-{\Xi}^-{\sim} & D_{\mathbb{C}^{\ast}}(Y_2, 0). 
}
\end{align}
Moreover we have the following: 
\begin{thm}\emph{(\cite{MR2593258, MR3581302})}\label{thm:FF}
For $\qQ \in D_G^b(Y_1 \times_T Y_2)$, suppose that the functor
\begin{align*}
\Phi^{\mathrm{forg}^G(\qQ)} \colon D^b(Y_1) \to D^b(Y_2)
\end{align*}
is fully-faithful (resp.~equivalence). Then
for the object
\begin{align*}
\widetilde{\qQ} \cneq 
i_{\ast} \circ \mathrm{forg}^{\mathbb{C}^{\ast}} \circ 
\Xi(\qQ) \in D_{G}(Y_1 \times Y_2, \chi, -p_1^{\ast}w_1+p_2^{\ast}w_2)
\end{align*}
the functor
\begin{align*}
\Psi^{\widetilde{\qQ}} \colon 
D_G(Y_1, \chi, w_1) \to D_G(Y_2, \chi, w_2)
\end{align*}
is fully-faithful (resp.~equivalence). 
\end{thm}

\subsection{{K}n\"orrer periodicity}
Let $\eE \to Y$ be an algebraic vector bundle on a regular 
$\mathbb{C}$-scheme $Y$, and 
$s \colon Y \to \eE$ be a regular section of it, i.e. its zero locus 
\begin{align*}
Z\cneq (s=0) \subset Y
\end{align*}
has codimension equals to the rank of $\eE$. 
The section $s$ naturally defines the morphism
\begin{align*}
Q_s \colon 
\eE^{\vee} \to \mathbb{A}^1
\end{align*}
by sending $(y, v)$ for $y \in Y$ and $v \in \eE|_{y}^{\vee}$ 
to $\langle s(y), v \rangle$. 
We have the following diagram
\begin{align*}
\xymatrix{
\eE|_{Z}^{\vee} \ar@<-0.3ex>@{^{(}->}[r]^-{i} \ar[d]_-{p} & \eE^{\vee} \ar[r]^-{Q_s} \ar[d] & \mathbb{A}^1 \\
Z \ar@<-0.3ex>@{^{(}->}[r] & Y. &
}
\end{align*}
Note that $Q_s=0$ on $i(\eE|_{Z}^{\vee}) \subset \eE^{\vee}$. 
Let $\mathbb{C}^{\ast}$ acts on $Z$ trivially, 
and on $\eE^{\vee}$ with weight one 
on the fibers of the projection 
$\eE^{\vee} \to Y$. 
The following is the version of {K}n\"orrer periodicity
used in this paper: 
\begin{thm}\emph{(\cite{MR3071664, MR2982435, MR3631231})}\label{thm:knoer}
The functor 
\begin{align*}
i_{\ast} \circ p^{\ast} \colon D_{\mathbb{C}^{\ast}}(Z, 0) \to D_{\mathbb{C}^{\ast}}(\eE^{\vee}, Q_s)
\end{align*}
is an equivalence of triangulated categories. 
By composing it 
with the equivalence (\ref{equiv:w=0}), we obtain the
equivalence
\begin{align*}
i_{\ast} \circ p^{\ast} \circ \Xi \colon 
D^b(Z) \stackrel{\sim}{\to} D_{\mathbb{C}^{\ast}}(\eE^{\vee}, Q_s).
\end{align*}
\end{thm}

\section{SOD via simple flips}\label{sec:sod}
Let $\widehat{U}$ be the formal completion of an affine 
space at the origin.  
In this section, we show that 
for a simple d-critical flip 
\begin{align*}
\widehat{M}^{+} \to \widehat{U} \leftarrow \widehat{M}^{-}
\end{align*}
 for smooth schemes 
$\widehat{M}^{\pm}$
satisfying some conditions, we have 
the SOD
\begin{align}\label{SOD:Mpm:U}
D^b(\widehat{M}^{+})=\langle \overbrace{D^b(\widehat{U}), \ldots, D^b(\widehat{U})}^{n}, D^b(\widehat{M}^{-}) \rangle.
\end{align}
The result is proved by combining derived factorization analogue 
of Bondal-Orlov's SOD associated with simple flips~\cite{B-O2} 
(see Theorem~\ref{thm:deq})
with the {K}n\"orrer periodicity
of derived factorization categories 
(see~Theorem~\ref{thm:knoer}). 

The main ingredient in this section is 
to show that the
kernel object of the fully-faithful functor
$D^b(\widehat{M}^{-}) \hookrightarrow D^b(\widehat{M}^{+})$ 
is given by the structure 
sheaf of the fiber product 
$\widehat{M}^{+} \times_{\widehat{U}} \widehat{M}^{-}$. 
This explicit description of the kernel 
will be important in the next section 
to globalize
the result in this section. 
 
\subsection{Simple toric flips}\label{subsec:flip}
Let $V^+$, $V^-$ be 
$\mathbb{C}$-vector spaces with dimensions $a$, $b$
respectively. 
We assume that $a\ge b$, and set 
\begin{align*}
n\cneq a-b \ge 0. 
\end{align*}
Let $\mathbb{C}^{\ast}$ acts on $V^+$, $V^-$
by weight $1$, $-1$ respectively. We 
fix bases of $V^{\pm}$ and 
denote the 
coordinates
of $V^+$, $V^-$ by 
\begin{align*}
\vec{x}=(x_1, \ldots, x_a), \ 
\vec{y}=(y_1, \ldots, y_b)
\end{align*}
respectively. 
Let $V^{\pm \ast} \cneq V^{\pm} \setminus \{0\}$, and 
define
 $Y^+$, $Y^-$ and $Z$ to be
\begin{align}\label{def:YZ}
&Y^+ \cneq \left[ \left((V^{+ \ast}) \times V^- \right)/\mathbb{C}^{\ast} \right]
=\mathrm{Tot}_{\mathbb{P}(V^+)}(\oO_{\mathbb{P}(V^+)}(-1) \otimes V^-)
 \\
\notag
&Y^- \cneq \left[ \left( V^+ \times (V^{- \ast}) \right)/\mathbb{C}^{\ast} \right]
=\mathrm{Tot}_{\mathbb{P}(V^-)}(\oO_{\mathbb{P}(V^-)}(-1) \otimes V^+) \\
\notag
&Z \cneq (V^+ \times V^-) \sslash \mathbb{C}^{\ast}
=\Spec \mathbb{C}[x_i y_j : 1\le i\le a, 1\le j\le b]. 
\end{align}
We have the toric flip diagram, called
\textit{simple flip} (see~\cite{Rei})
\begin{align}\notag
\xymatrix{
Y^+ \ar@{.>}[rr]^-{\phi} \ar[rd]_-{f^+} & & Y^- \ar[ld]^-{f^-} \\
& Z. &
}
\end{align}
We also have the projections and closed embeddings
\begin{align}\label{emb:closed}
\mathrm{pr}^{\pm} \colon Y^{\pm} \to \mathbb{P}(V^{\pm}), \ 
i^{\pm} \colon 
\mathbb{P}(V^{\pm}) \hookrightarrow Y^{\pm}
\end{align}
where $i^{\pm}$ are the zero sections of $\mathrm{pr}^{\pm}$. 

By setting $W \cneq Y^+ \times_{Z} Y^-$, 
we have the following diagram
\begin{align}\label{dia:blow}
\xymatrix{
& W \ar[ld]_-{p^+} \ar[rd]^-{p^-} & \\
Y^+ & & Y^-
}
\end{align}
where $p^{\pm}$ are the projections. 
Note that $p^{\pm}$ are the blow-ups 
of $Y^{\pm}$
at the smooth loci
$i^{\pm}(\mathbb{P}(V^{\pm}))$. 
The fiber product $W$ is also described as 
\begin{align}\notag
W 
&=\mathrm{Tot}_{\mathbb{P}(V^+) \times \mathbb{P}(V^-)}(\oO_{\mathbb{P}(V^+) \times \mathbb{P}(V^-)}(-1, -1)) \\
\label{W:quot}
&=\left[ \left( (V^{+\ast} ) \times (V^{-\ast}) 
\times \mathbb{C} \right)/(\mathbb{C}^{\ast})^2  \right].
\end{align}
Here $(s_1, s_2) \in (\mathbb{C}^{\ast})^2$ acts on 
$V^{+} \times V^{-} \times \mathbb{C}$ by 
\begin{align*}
(s_1, s_2) \cdot (\vec{x}, \vec{y}, t)=(s_1 \vec{x},  s_2^{-1} \vec{y}, s_1^{-1} s_2 t). 
\end{align*}
Under the description of $W$ in (\ref{W:quot}), the projections
$p^{\pm} \colon W \to Y^{\pm}$ are induced by
maps
\begin{align*}
V^+ \times V^- \times \mathbb{C} \to V^+ \times V^-, \ 
(\mathbb{C}^{\ast})^2 \to \mathbb{C}^{\ast}
\end{align*}
defined by 
\begin{align*}
&p^+ \colon (\vec{x}, \vec{y}, t) \mapsto (\vec{x}, t\vec{y}), \ (s_1, s_2) \mapsto s_1, \\
&p^- \colon (\vec{x}, \vec{y}, t) \mapsto (t\vec{x}, \vec{y}), \ (s_1, s_2) 
\mapsto s_2
\end{align*}
respectively. 
Let $s \in \mathbb{C}^{\ast}$ acts on 
$Y^+$, $Y^-$, $W$
by
\begin{align}\label{act:C}
s \cdot (\vec{x}, \vec{y})=(\vec{x}, s\vec{y}), \ 
s \cdot (\vec{x}, \vec{y})=(s\vec{x}, \vec{y}), \ 
s \cdot (\vec{x}, \vec{y}, t)=(\vec{x}, \vec{y}, st)
\end{align}
respectively. 
Then the diagram (\ref{dia:blow}) is equivariant 
with respect to the above $\mathbb{C}^{\ast}$-actions.

\subsection{Critical loci}\label{subsec:crit}
Let $\widehat{U}$ be a smooth 
$\mathbb{C}$-scheme of dimension $g$, 
given by
\begin{align*}
\widehat{U} \cneq \Spec \mathbb{C}[[u_1, \ldots, u_g]].
\end{align*}
Let us take  
an element $w \in \Gamma(\oO_{Z_{\widehat{U}}})$ written as 
\begin{align}\label{def:g}
w=\sum_{i=1}^a \sum_{j=1}^b x_i y_j w_{ij}(\vec{u})
\end{align}
for some $w_{ij}(\vec{u}) \in \Gamma(\oO_{\widehat{U}})$. 
We consider the following commutative diagram
\begin{align}\label{def:pm0}
\xymatrix{
& W_{\widehat{U}} \ar[ld]_-{p_{\widehat{U}}^+} \ar[rd]^-{p_{\widehat{U}}^-} & \\
Y^{+}_{\widehat{U}} \ar[rd]_-{f_{\widehat{U}}^+}
\ar@/_15pt/[rdd]_-{w^{+}}  \ar@{.>}[rr]^-{\phi_{\widehat{U}}}& &
 Y^{-}_{\widehat{U}} 
\ar[ld]^-{f_{\widehat{U}}^-}
\ar@/^15pt/[ldd]^-{w^-}\\
& Z_{\widehat{U}} \ar[d]_-{w} &  \\
& \mathbb{A}^1 &
}
\end{align}
Then the composition 
\begin{align*}
\widetilde{w} \cneq p_{\widehat{U}}^{+\ast}w^{+}=p_{\widehat{U}}^{-\ast}w^- \colon W_{\widehat{U}} \to \mathbb{A}^1
\end{align*}
is written as 
\begin{align}\label{w:tilde}
\widetilde{w}=t \sum_{i, j} x_i y_j w_{ij}(\vec{u})
\end{align}
in the description of $W$ by (\ref{W:quot}). 
We define $\widehat{M}^{\pm}$ to be
\begin{align*}
\widehat{M}^{\pm} \cneq 
\{dw^{\pm}=0\}
\subset Y_{\widehat{U}}^{\pm}.
\end{align*}

\begin{lem}\label{lem:Msmooth}
Suppose that 
$\widehat{M}^{\pm}$ are smooth and irreducible of dimension 
\begin{align}\label{dim:M}
\dim \widehat{M}^{+}=n+g-1, \ \dim \widehat{M}^{-}=-n+g-1
\end{align}
respectively. 
Then 
$\widehat{M}^{\pm}$ are contained in the 
images of 
$i_{\widehat{U}}^{\pm} \colon \mathbb{P}(V^{\pm})_{\widehat{U}} \hookrightarrow
Y_{\widehat{U}}^{\pm}$, 
where $i^{\pm}$ are given in (\ref{emb:closed}). 
Moreover  
we have
\begin{align}\label{write:M}
&\widehat{M}^{+}=\left\{(\vec{x}, \vec{u}) 
\in \mathbb{P}(V^{+})_{\widehat{U}} : 
\sum_{i=1}^a x_i w_{ij}(\vec{u})=0 \mbox{ for all }
1\le j\le b\right\}, \\
\notag
&\widehat{M}^{-}
=\left\{(\vec{y}, \vec{u}) \in \mathbb{P}(V^{-})_{\widehat{U}} : 
\sum_{j=1}^b y_j w_{ij}(\vec{u})=0 \mbox{ for all }
1\le i\le a\right\}.
\end{align}
\end{lem}
\begin{proof}
Let $N^{+}$ be the scheme defined by the RHS of (\ref{write:M}). 
Note that we obviously have the closed embedding
\begin{align}\label{emb:NM}
N^{+} \hookrightarrow \widehat{M}^{+}, \ 
(\vec{x}, \vec{u}) \mapsto (\vec{x}, \vec{y}=0, \vec{u}).
\end{align}
Since $N^+$ is defined by $b$-equations on 
the smooth scheme $\mathbb{P}(V^+)_{\widehat{U}}$
of dimension $a+g-1$, 
we have 
\begin{align*}
\dim N^+ \ge a+g-1-b=n+g-1.
\end{align*}
Therefore the assumption on $\widehat{M}^{+}$ implies that 
the embedding (\ref{emb:NM}) is an isomorphism. 
The claim for $\widehat{M}^{-}$ is similarly proved. 
\end{proof}

\begin{rmk}\label{rmk:prepared}
The assumption of Lemma~\ref{lem:Msmooth}
is satisfied if 
$g=ab$ and $w_{ij}(\vec{u})=u_{ij}$
where $\{u_{ij}\}_{1\le i\le a, 1\le j\le b}$ 
is a coordinate system of $\widehat{U}$. 
If the assumption of Lemma~\ref{lem:Msmooth}
is satisfied, then the projections
$\widehat{M}^{\pm} \to \widehat{U}$ are 
well-presented families of projective spaces
defined in~\cite[Section~3]{MR0349687}. 
\end{rmk}

Under the assumption of Lemma~\ref{lem:Msmooth}, 
we have
 $f_{\widehat{U}}^{\pm}(\widehat{M}^{\pm}) \subset \{0\} \times \widehat{U}$. 
Therefore 
$\pi^{\pm} \cneq (f^{\pm}_{\widehat{U}})|_{\widehat{M}^{\pm}}$
induces the diagram
\begin{align}\label{dia:pi0}
\xymatrix{
\widehat{M}^{+} \ar[rd]_-{\pi^+} & & \widehat{M}^{-} \ar[ld]^-{\pi^-} \\
& \widehat{U}. &
}
\end{align}
Moreover for each $c \in \mathbb{Z}$, 
we have the line bundles
\begin{align*}
\oO_{\widehat{M}^{\pm}}(c) \cneq 
\oO_{\mathbb{P}(V^{\pm})_{\widehat{U}}}(c)|_{\widehat{M}^{\pm}}
\in \Pic(\widehat{M}^{\pm}). 
\end{align*}

\subsection{SOD of derived factorization categories under simple flips}
Let us consider the diagram (\ref{def:pm0}). 
Since $w^+$, $w^-$ and $\widetilde{w}$ are of
weight one with respect to the 
$\mathbb{C}^{\ast}$-actions (\ref{act:C}), 
we have the 
associated derived factorization categories
\begin{align*}
D_{\mathbb{C}^{\ast}}(Y_{\widehat{U}}^+, w^+), \ 
D_{\mathbb{C}^{\ast}}(Y_{\widehat{U}}^-, w^-), \ 
D_{\mathbb{C}^{\ast}}(W_{\widehat{U}}, \widetilde{w})
\end{align*}
respectively. 
Since the diagram (\ref{dia:blow}) is $\mathbb{C}^{\ast}$-equivariant, 
we have the functors
\begin{align*}
&\dL p_{\widehat{U}}^{-\ast} \colon 
D_{\mathbb{C}^{\ast}}(Y^-_{\widehat{U}}, w^-)
\to D_{\mathbb{C}^{\ast}}(W_{\widehat{U}}, \widetilde{w}), \\ 
&\dR p_{\widehat{U}\ast}^+ \colon 
D_{\mathbb{C}^{\ast}}(W_{\widehat{U}}, \widetilde{w})
\to 
D_{\mathbb{C}^{\ast}}(Y_{\widehat{U}}^+, w^+).
\end{align*}
By composing them, we obtain the functor
\begin{align}\label{def:PhiY}
\Psi_Y \cneq \dR p_{\widehat{U}\ast}^+ \circ \dL p_{\widehat{U}}^{-\ast} \colon 
D_{\mathbb{C}^{\ast}}(Y^-_{\widehat{U}}, w^-)
\to
D_{\mathbb{C}^{\ast}}(Y_{\widehat{U}}^+, w^+).
\end{align}
Let
\begin{align*}
g \colon \mathbb{P}(V^+)_{\widehat{U}} \to \widehat{U}, \
i_{\widehat{U}}^+ \colon \mathbb{P}(V^+)_{\widehat{U}} \hookrightarrow 
Y_{\widehat{U}}^+
\end{align*}
 be the projection, 
the inclusion into the zero section (\ref{emb:closed})
respectively.  
We also have the functor
\begin{align}\label{def:ups}
\Upsilon_Y \cneq i^{+}_{\widehat{U}\ast} \circ \dL g^{\ast} \colon
 D_{\mathbb{C}^{\ast}}(\widehat{U}, 0) \to 
D_{\mathbb{C}^{\ast}}(Y_{\widehat{U}}^+, w^+).
\end{align}
Here $\mathbb{C}^{\ast}$ acts on 
$\widehat{U}$, $\mathbb{P}(V^+)_{\widehat{U}}$ trivially. 
The following result should be well-known, 
but we include a proof here as we cannot find a reference. 
\begin{thm}\label{thm:deq}
(i) 
The functor $\Psi_Y$ in (\ref{def:PhiY}) 
is fully-faithful. 

(ii) If $n\ge 1$, the functor 
$\Upsilon_Y$ in (\ref{def:ups}) is fully-faithful. 

(iii)
By setting $\Upsilon_Y^i \cneq \otimes \oO_{Y_{\widehat{U}}^+}(i) \circ \Upsilon_Y$, 
we have 
the SOD
\begin{align*}
D_{\mathbb{C}^{\ast}}(Y_{\widehat{U}}^+, w^+)
=\langle \Imm \Upsilon_Y^{-n}, \ldots, \Imm \Upsilon_Y^{-1}, 
\Imm \Psi_Y \rangle. 
\end{align*}
\end{thm}
\begin{proof}
(i) The functor $\Psi_Y$ is written as
$\Psi^{\widetilde{\oO}_W}$ 
in the notation of Theorem~\ref{thm:FF}. 
On the other hand, the 
functor 
\begin{align}\label{funct:PhiW}
\Phi^{\oO_W} \colon D^b(Y_{\widehat{U}}^-) \to D^b(Y_{\widehat{U}}^+)
\end{align}
is fully-faithful by~\cite{B-O2}. 
Therefore (i) follows from Theorem~\ref{thm:FF}.
The proof of (ii) is similar. 

We prove (iii). 
Let us recall that we 
have a similar SOD 
using \textit{windows}~\cite{MR3327537, BFK2}. 
Let $T_j=\mathbb{C}^{\ast}$ for $j=1, 2$
acts on $V^+ \times V^- \times \widehat{U}$ by 
weight $(1, -1, 0)$ for $j=1$, 
and $(1, 0, 0)$ for $j=2$. 
Then the open immersions
\begin{align}\label{eta:open}
\eta^{\pm} \colon Y_{\widehat{U}}^{\pm} \hookrightarrow 
[(V^+ \times V^-)_{\widehat{U}}/T_1]
\end{align}
are $\mathbb{C}^{\ast}$-equivariant, where 
the $\mathbb{C}^{\ast}$-action on the LHS 
is given by (\ref{act:C}) and 
that on the RHS is given by the 
above $T_2$-action. 
Let $\oO_{\bullet}(i)$ be the 
$T_1$-equivariant line bundle on 
$\Spec \mathbb{C}$, 
given by a one dimensional $T_1$-representation  
with weight $i$. 
We denote by $\oO_{(V^+ \times V^-)_{\widehat{U}}}(i)$
the pull-back of $\oO_{\bullet}(i)$ under the 
structure morphism 
\begin{align*}
(V^+ \times V^-)_{\widehat{U}} \to \Spec \mathbb{C}.
\end{align*}
Let $\chi \colon T_1 \times T_2 \to \mathbb{C}^{\ast}$
be the second projection, and take
$w \in \Gamma(\oO_{(V^+ \times V^-)_{\widehat{U}}})$ as in (\ref{def:g}). 
For a subset $I \subset \mathbb{R}$, 
the window subcategory 
\begin{align}\label{window}
\wW_I \subset D_{T_1 \times T_2}((V^+ \times V^-)_{\widehat{U}}, \chi ,w)
\end{align}
is defined 
to be 
the thick triangulated subcategory generated by 
the factorizations (\ref{factorization}) 
where $\fF_0$, $\fF_1$ are of the form
\begin{align*}
\fF_j=\bigoplus_{-i \in I \cap \mathbb{Z}} 
\oO_{(V^+ \times V^-)_{\widehat{U}}}(i)^{\oplus l_{i, j}}, \ j=0, 1, \ 
l_{i, j} \in \mathbb{Z}_{\ge 0}
\end{align*}
as $T_1$-equivariant sheaves.
Here we regard $(T_1 \times T_2)$-equivariant 
sheaves $\fF_j$ as $T_1$-equivariant 
sheaves by the inclusion 
$T_1 \hookrightarrow T_1 \times T_2$, 
$t_1 \mapsto (t_1, 1)$. 
By~\cite[Theorem~3.5.2]{BFK2}, 
there exists a fully-faithful functor
\begin{align*}
\Psi_Y' \colon D_{\mathbb{C}^{\ast}}(Y^-_{\widehat{U}}, w^-)
\to
D_{\mathbb{C}^{\ast}}(Y_{\widehat{U}}^+, w^+)
\end{align*}
which fits into the following commutative diagram
\begin{align}\label{dia:Psi'}
\xymatrix{
\wW_{(-b, 0]} \ar[r]^-{\eta^{-\ast}}_-{\sim} \ar[d]  & D_{\mathbb{C}^{\ast}}(Y_{\widehat{U}}^-, w^-) \ar[d]^-{\Psi_Y'} \\
\wW_{(-b, a-b]} \ar[r]_-{\eta^{+\ast}}^-{\sim} & 
D_{\mathbb{C}^{\ast}}(Y_{\widehat{U}}^+, w^+).
}
\end{align}
Here the horizontal arrows are equivalences of
triangulated categories, 
defined by pull-backs via open 
immersions (\ref{eta:open}) restricted to 
$\wW_I$, and 
the left vertical arrow is a natural 
inclusion. 
Moreover by \textit{loc.~cit.~}, we have the SOD
\begin{align*}
D_{\mathbb{C}^{\ast}}(Y_{\widehat{U}}^+, w^+)
=\langle \Imm \Upsilon_Y^{-n}, \ldots, \Imm \Upsilon_Y^{-1}, 
\Imm \Psi_Y' \rangle. 
\end{align*}
It is enough to show that 
$\Imm \Psi_Y=\Imm \Psi_Y'$. 
Note that we have
\begin{align*}
\eta^{-\ast}(\oO_{(V^+ \times V^-)_{\widehat{U}}}(i))=\oO_{Y_{\widehat{U}}^-}(-i), \ 
\eta^{+\ast}(\oO_{(V^+ \times V^-)_{\widehat{U}}}(i))
=\oO_{Y_{\widehat{U}}^+}(i).
\end{align*}
By the diagram (\ref{dia:Psi'}),
it follows that $\Imm \Psi_Y'$ is 
generated by 
factorizations (\ref{factorization})
such that 
$\fF_0$, $\fF_1$ are of the form
\begin{align*}
\fF_j=\bigoplus_{0 \le i\le b-1} \oO_{Y_{\widehat{U}}^+}(-i)^{\oplus l_{i, j}}, \ 
j=0, 1, \ l_{i, j} \in \mathbb{Z}_{\ge 0}. 
\end{align*}
On the other hand, 
an easy calculation shows that
\begin{align}\label{fact:Phi}
\Phi^{\oO_W}(\oO_{Y_{\widehat{U}}^-}(i))=\oO_{Y_{\widehat{U}}^+}(-i), \ 
0 \le i\le b-1
\end{align}
where $\Phi^{\oO_W}$ is the functor (\ref{funct:PhiW}). 
Together with the equivalence of the top horizontal arrow of (\ref{dia:Psi'}), 
it follows that $\Imm \Psi_Y$ is also generated by the objects of the form (\ref{fact:Phi}). 
Therefore $\Imm \Psi_Y=\Imm \Psi_Y'$ holds. 

\end{proof}

\subsection{SOD in the complete local setting}
We return to the situation in Subsection~\ref{subsec:crit}. 
Under the assumption of Lemma~\ref{lem:Msmooth}, we have the following diagram
\begin{align*}
\xymatrix{
A^{\pm} \ar[d]_-{q^{\pm}}
\ar@<-0.3ex>@{^{(}->}[r]^-{j^{\pm}} 
\ar@{}[dr]|\square
& Y_{\widehat{U}}^{\pm} \ar[r]^{w^{\pm}} 
\ar[d]^-{\mathrm{pr}_{\widehat{U}}^{\pm}} & \mathbb{A}^1 \\
\widehat{M}^{\pm} \ar@<-0.3ex>@{^{(}->}[r] & \mathbb{P}(V^{\pm})_{\widehat{U}}
}
\end{align*}
Here 
$A^{\pm}$ are defined by the 
above Cartesian square. 
By Theorem~\ref{thm:knoer},
the above diagram induces the equivalence
\begin{align}\label{equiv:01}
j_{\ast}^{\pm} \circ q^{\pm \ast} \colon 
D_{\mathbb{C}^{\ast}}(\widehat{M}^{\pm}, 0) \stackrel{\sim}{\to}
D_{\mathbb{C}^{\ast}}(Y_{\widehat{U}}^{\pm}, w^{\pm}).
\end{align} 
Here $\mathbb{C}^{\ast}$ acts on $M^{\pm}$ trivially, and 
on $Y_{\widehat{U}}^{\pm}$
with weight one on fibers of $\mathrm{pr}_{\widehat{U}}^{\pm}$. 
Let $B \subset (\mathbb{P}(V^+) \times \mathbb{P}(V^-))_{\widehat{U}}$
be defined by
\begin{align*}
B\cneq \left\{ (\vec{x}, \vec{y}, \vec{u}) 
\in (\mathbb{P}(V^+) \times \mathbb{P}(V^-))_{\widehat{U}} 
:
\sum_{i, j} x_i y_j w_{ij}(\vec{u})=0
 \right\}.
\end{align*}
Similarly, we have the following diagram
\begin{align*}
\xymatrix{
E\ar[d]_-{\widetilde{q}}
\ar@<-0.3ex>@{^{(}->}[r]^-{\widetilde{j}} 
\ar@{}[dr]|\square
& W_{\widehat{U}} \ar[r]^{\widetilde{w}} \ar[d]^-{\mathrm{pr}_{\widehat{U}}}
 & \mathbb{A}^1  \\
B \ar@<-0.3ex>@{^{(}->}[r] & (\mathbb{P}(V^{+}) \times 
\mathbb{P}(V^-))_{\widehat{U}}
}
\end{align*}
where
the right vertical arrow is the projection 
and $E$ is defined by the above Cartesian square. 
Again by Theorem~\ref{thm:knoer}
and the description of $\widetilde{w}$ in (\ref{w:tilde}), 
the above diagram induces the equivalence
\begin{align}\label{equiv:02}
\widetilde{j}_{\ast} \circ \widetilde{q}^{\ast} \colon 
D_{\mathbb{C}^{\ast}}(B, 0) \stackrel{\sim}{\to}
D_{\mathbb{C}^{\ast}}(W_{\widehat{U}}, \widetilde{w}).
\end{align} 
Here $\mathbb{C}^{\ast}$ acts on $B$ trivially, and
on $W_{\widehat{U}}$ with weight one 
on fibers of $\mathrm{pr}_{\widehat{U}}$. 
Let $F^{\pm} \subset (\mathbb{P}(V^+) \times \mathbb{P}(V^-))_{\widehat{U}}$
be defined by
\begin{align}\notag
&F^{+} \cneq \left\{(\vec{x}, \vec{y}, \vec{u}) \in 
(\mathbb{P}(V^+) \times \mathbb{P}(V^-))_{\widehat{U}} : 
\sum_{i=1}^a x_i w_{ij}(\vec{u})=0 \mbox{ for all }
1\le j\le b\right\}, \\
\notag
&F^- \cneq \left\{(\vec{x}, \vec{y}, \vec{u}) \in 
(\mathbb{P}(V^+) \times \mathbb{P}(V^-))_{\widehat{U}} : 
\sum_{j=1}^b y_j w_{ij}(\vec{u})=0 \mbox{ for all }
1\le i\le a\right\}.
\end{align}
Note that we have 
$F^{\pm} \subset B$. 
Also the projections
$(\mathbb{P}(V^+) \times \mathbb{P}(V^-))_{\widehat{U}} \to 
\mathbb{P}(V^{\pm})_{\widehat{U}}$
restricted to $F^{\pm}$ give 
morphisms 
$F^{\pm} \to \widehat{M}^{\pm}$, which are 
trivial $\mathbb{P}(V_{\mp})$-bundles. 
So we have the diagram
\begin{align*}
\xymatrix{
F^{\pm} 
\ar@<-0.3ex>@{^{(}->}[r]^-{k^{\pm}} \ar[d]_-{r^{\pm}} & B \\
\widehat{M}^{\pm}.  &
} 
\end{align*}
Let $\Theta^{\pm}$ be the functor defined by
\begin{align*}
\Theta^{\pm} \cneq k^{\pm}_{\ast} \circ 
r^{\pm \ast}  \colon 
D_{\mathbb{C}^{\ast}}(\widehat{M}^{\pm}, 0) \to D_{\mathbb{C}^{\ast}}(B, 0). 
\end{align*}
\begin{lem}\label{lem:commute1}
The following diagram is commutative: 
\begin{align}\label{dia:commute}
\xymatrix{
D_{\mathbb{C}^{\ast}}(\widehat{M}^{\pm}, 0) \ar[r]^-{\Theta^{\pm}} 
\ar[d]_-{j^{\pm}_{\ast}\circ q^{\pm \ast}}&
D_{\mathbb{C}^{\ast}}(B, 0) 
\ar[d]^-{\widetilde{j}_{\ast} \circ \widetilde{q}^{\ast}} \\
D_{\mathbb{C}^{\ast}}(Y_{\widehat{U}}^{\pm}, w^{\pm}) \ar[r]_-{\dL p_{\widehat{U}}^{\pm \ast}} & 
D_{\mathbb{C}^{\ast}}(W_{\widehat{U}}, \widetilde{w}).
}
\end{align}
Here the vertical arrows are equivalences (\ref{equiv:01}), (\ref{equiv:02}). 
\end{lem}
\begin{proof}
Let $\widetilde{F}^{\pm} \subset W_{\widehat{U}}$ be defined by
the Cartesian square
\begin{align}\notag
\xymatrix{
\widetilde{F}^{\pm} \ar@<-0.3ex>@{^{(}->}[r]^-{\widetilde{i}^{\pm}} \ar[d] 
\ar@{}[dr]|\square
& W_{\widehat{U}} \ar[d]^-{\mathrm{pr}_{\widehat{U}}} \\
F^{\pm} \ar@<-0.3ex>@{^{(}->}[r] & 
(\mathbb{P}(V^+) \times \mathbb{P}(V^-))_{\widehat{U}}.
}
\end{align}
Here the right vertical arrow is the projection. 
We have two diagrams
\begin{align}\notag
\xymatrix{
\widetilde{F}^{\pm} \ar@<-0.3ex>@{^{(}->}[r]^-{\widetilde{i}^{\pm}}
\ar@{}[dr]|\square
\ar[d] 
\ar@/_15pt/[dd]_-{\widetilde{q}^{\pm}} & W_{\widehat{U}} 
\ar[d]^-{p_{\widehat{U}}^{\pm}} \\
A^{\pm} \ar@<-0.3ex>@{^{(}->}[r]^-{j^{\pm}} \ar[d]^-{q^{\pm}} 
\ar@{}[dr]|\square
& Y_{\widehat{U}}^{\pm} \ar[d]^-{\mathrm{pr}_{\widehat{U}}^{\pm}} \\
\widehat{M}^{\pm}  \ar@<-0.3ex>@{^{(}->}[r] & \mathbb{P}(V^{\pm})_{\widehat{U}},
}
\quad
\xymatrix{
\widetilde{F}^{\pm} \ar@/_15pt/[dd]_-{\widetilde{q}^{\pm}}
\ar@/^15pt/[rr]^-{\widetilde{i}^{\pm}} 
 \ar@<-0.3ex>@{^{(}->}[r]  \ar[d] 
\ar@{}[dr]|\square &
E \ar@<-0.3ex>@{^{(}->}[r]^-{\widetilde{j}}  \ar[d]^-{\widetilde{q}}
\ar@{}[dr]|\square 
& W_{\widehat{U}} \ar[d]^-{\mathrm{pr}_{\widehat{U}}} \\
F^{\pm} \ar@<-0.3ex>@{^{(}->}[r]^-{k^{\pm}} \ar[d]^-{r^{\pm}}
\ar@{}[drr]|\square 
 &
B \ar@<-0.3ex>@{^{(}->}[r]
& 
(\mathbb{P}(V^+) \times \mathbb{P}(V^-)_{\widehat{U}} \ar[d] \\
\widehat{M}^{\pm}  \ar@<-0.3ex>@{^{(}->}[rr] & & 
\mathbb{P}(V^{\pm})_{\widehat{U}}.
}
\end{align}
Since every Cartesians in the above diagrams 
are derived Cartesians, the base change 
shows that 
\begin{align*}
\dL p_{\widehat{U}}^{\pm \ast} \circ j_{\ast}^{\pm} \circ q^{\pm \ast}
\cong \widetilde{i}^{\pm}_{\ast} \circ 
\widetilde{q}^{\pm \ast} \cong 
\widetilde{j}_{\ast} \circ \widetilde{q}^{\ast} \circ \Theta^{\pm}.
\end{align*}
Therefore
the lemma holds. 
\end{proof}

\begin{lem}\label{lem:commute2}
The following diagram is commutative
\begin{align}\label{dia:commute2}
\xymatrix{
D_{\mathbb{C}^{\ast}}(\widehat{M}^{\pm}, 0)  
\ar[d]_-{j^{\pm}_{\ast}\circ q^{\pm \ast}}&
D_{\mathbb{C}^{\ast}}(B, 0) \ar[l]_-{\Theta_R^{\pm}}
\ar[d]^-{\widetilde{j}_{\ast} \circ \widetilde{q}^{\ast}} \\
D_{\mathbb{C}^{\ast}}(Y_{\widehat{U}}^{\pm}, w^{\pm})  & 
D_{\mathbb{C}^{\ast}}(W_{\widehat{U}}, \widetilde{w}) 
\ar[l]^-{\dR p_{\widehat{U}\ast}^{\pm}}.
}
\end{align}
Here $\Theta_R^{\pm}$ are 
the right adjoint functors of $\Theta$, i.e. 
\begin{align*}
\Theta_R^{\pm} \cneq \dR r_{\ast}^{\pm} \circ (k^{\pm})^{!},
\end{align*}
where $(k^{\pm})^{!}$ are the right adjoint functors of 
$k^{\pm}_{\ast}$. They are written as
\begin{align}\label{k!}
(k^{+})^{!} &=\otimes \oO_{F^+}(b-1, -1) \circ \dL k^{+ \ast}[1-b], \\
\notag
(k^{-})^{!} &=\otimes \oO_{F^-}(-1, a-1) \circ \dL k^{- \ast}[1-a].
\end{align}
Here $\oO_{F^{\pm}}(c, d) \cneq 
\oO_{(\mathbb{P}(V^+) \times \mathbb{P}(V^-))_{\widehat{U}}}(c, d)|_{F^{\pm}}$. 
\end{lem}
\begin{proof}
The commutativity of (\ref{dia:commute2}) follows 
from that of (\ref{dia:commute}) together with the 
fact that $\dR p_{\widehat{U}\ast}^{\pm}$, 
$\Theta^{\pm}_R$ are the right adjoint functors of 
$\dL p_{\widehat{U}}^{\pm \ast}$, $\Theta^{\pm}$ respectively. 
As for the formula for $(k^{\pm})^{!}$,
note that we have 
\begin{align*}
(k^{\pm})^{!}(-)
=\otimes \det N_{F^{\pm}/B} \circ \dL k^{\pm \ast}(-)
\circ 
[\dim F^{\pm}-\dim B].
\end{align*}
By the exact sequences
\begin{align*}
0 \to N_{F^{\pm}/B} \to 
N_{F^{\pm}/(\mathbb{P}(V^+) \times \mathbb{P}(V^-))_{\widehat{U}}}
\to N_{B/(\mathbb{P}(V^+) \times \mathbb{P}(V^-))_{\widehat{U}}} \to 0
\end{align*}
we have 
\begin{align*}
\det N_{F^{+}/B}=\oO_{F^+}(b-1, -1), \ 
\det N_{F^-/B}=\oO_{F^-}(a-1, -1).
\end{align*}
Together with the dimension computations
\begin{align}\label{dim:form}
\dim F^+=g+a-2, \ 
\dim F^-=g+b-2, \ 
\dim B=g+a+b-3
\end{align}
we obtain (\ref{k!}). 
\end{proof}

\begin{prop}\label{prop:commute2}
Suppose that 
the following condition holds: 
\begin{align}\label{dim:fprod}
\dim (\widehat{M}^{+} \times_{\widehat{U}} \widehat{M}^{-}) \le g-1.
\end{align}
Then the following diagram is commutative: 
\begin{align}\notag
\xymatrix{
D_{\mathbb{C}^{\ast}}(\widehat{M}^{-}, 0)  \ar[r]^-{\Theta_{\widehat{M}}}
\ar[d]_-{j^{-}_{\ast}\circ q^{-\ast}}&
D_{\mathbb{C}^{\ast}}(\widehat{M}^{+}, 0) 
\ar[d]^-{j^{+}_{\ast}\circ q^{+\ast}} \\
D_{\mathbb{C}^{\ast}}(Y_{\widehat{U}}^{-}, w^{\pm}) \ar[r]_-{\Psi_Y} & 
D_{\mathbb{C}^{\ast}}(Y_{\widehat{U}}^+, \widetilde{w}) .
}
\end{align}
Here the vertical arrows are equivalences (\ref{equiv:01}), 
$\Psi_Y$ is given by (\ref{def:PhiY}) and 
$\Theta_{\widehat{M}}$ is defined by
\begin{align*}
\Theta_{\widehat{M}} \cneq 
\otimes \oO_{\widehat{M}^{+}}(b-1) \circ 
\Psi^{\Xi(\oO_{\widehat{M}^{+} \times_{\widehat{U}} \widehat{M}^{-}})}
\circ \otimes \oO_{\widehat{M}^{-}}(-1)[1-b]
\end{align*}
where $\Xi$ is the equivalence in (\ref{equiv:w=0}): 
\begin{align*}
\Xi \colon D^b(\widehat{M}^+ \times_{\widehat{U}} \widehat{M}^-)
\stackrel{\sim}{\to} D_{\mathbb{C}^{\ast}}(\widehat{M}^+ \times_{\widehat{U}} \widehat{M}^-, 0).
\end{align*}
\end{prop}
\begin{proof}
By Lemma~\ref{lem:commute1} and Lemma~\ref{lem:commute2}, 
it is enough to check that 
$\Theta_R^{+} \circ \Theta^-$
is isomorphic to $\Theta_{\widehat{M}}$. 
By setting $\oO_B(c, d) 
\cneq \oO_{(\mathbb{P}(V^+) \times \mathbb{P}(V^-))_{\widehat{U}}}
(c, d)|_{B}$, 
and using the formula (\ref{k!}), 
we have 
\begin{align*}
\Theta_R^+ \circ \Theta^-(-) &=
\dR r^+_{\ast} \left(\dL k^{+\ast}k^{-}_{\ast}r^{-\ast}(-) \otimes 
\oO_{F^+}(b-1, -1) \right)[1-b] \\
&=\dR r^+_{\ast} \dL k^{+\ast}\left(k^{-}_{\ast}r^{-\ast}(-) \otimes 
\oO_{B}(0, -1) \right)\otimes \oO_{\widehat{M}^+}(b-1)[1-b] \\
&=\dR r^+_{\ast} \dL k^{+\ast}k^{-}_{\ast} r^{-\ast}\left(- \otimes
\oO_{\widehat{M}^-}(-1)\right) \otimes \oO_{\widehat{M}^+}(b-1)[1-b].
\end{align*}
Let us consider the 
composition $\dL k^{+\ast} \circ k_{\ast}^-$ in the above formula. 
We have the following Cartesian 
diagram
\begin{align}\label{cart:F}
\xymatrix{
F^+ \cap F^- \ar@<-0.3ex>@{^{(}->}[r]^-{\theta^+} \ar@<-0.3ex>@{^{(}->}[d]_-{\theta^-} \ar@{}[dr]|\square & F^+  \ar@<-0.3ex>@{^{(}->}[d]^-{k^+} \\
F^-  \ar@<-0.3ex>@{^{(}->}[r]^-{k^-} & B.
}
\end{align}
By the definition of $F^{\pm}$, we have 
\begin{align}\label{F:intersect}
F^+ \cap F^- =\widehat{M}^{+} \times_{\widehat{U}} \widehat{M}^{-}.
\end{align}
Since $F^+ \subset B$ is of codimension 
$b-1$ and $F^- \subset B$ is of codimension 
$a-1$, we have 
\begin{align*}
\dim (F^+ \cap F^-) \ge \dim B-(b-1)-(a-1)=g-1. 
\end{align*}
Then the assumption of the dimension of the fiber product (\ref{dim:fprod})
implies that 
$\dim (F^+ \cap F^-)=g-1$ and 
the diagram (\ref{cart:F}) is a derived Cartesian. 
Therefore by the base change, we have 
\begin{align*}
\dL k^{+\ast} \circ k_{\ast}^- \cong \theta^{+}_{\ast} 
\circ \dL \theta^{-\ast}. 
\end{align*}
By substituting into the above formula for 
$\Theta_R^+ \circ \Theta^-$, and again noting 
(\ref{F:intersect}),  
we have
\begin{align*}
\Theta_R^+ \circ \Theta^-(-)
&=\dR r^+_{\ast} \theta^{+}_{\ast}
\dL \theta^{-\ast}
r^{-\ast}\left(- \otimes
\oO_{\widehat{M}^-}(-1)\right) \otimes \oO_{\widehat{M}^+}(b-1)[1-b] \\
&=\Psi^{\Xi(\oO_{\widehat{M}^{+} \times_{\widehat{U}} \widehat{M}^{-}})}
\left(- \otimes
\oO_{\widehat{M}^-}(-1)\right)\otimes \oO_{\widehat{M}^+}(b-1)[1-b] \\
&=\Theta_{\widehat{M}}(-).
\end{align*}
Therefore the proposition holds. 
\end{proof}

\begin{lem}\label{lem:commute3}
The following diagram is commutative
\begin{align}\label{dia:commute3}
\xymatrix{
D_{\mathbb{C}^{\ast}}(\widehat{U}, 0)
\ar[r]^-{\overline{\Upsilon}_{\widehat{M}}^{i+b}[-b]} \ar@{=}[d] & D_{\mathbb{C}^{\ast}}(\widehat{M}^{+}, 0) 
\ar[d]^-{j^{+}_{\ast}\circ q^{+\ast}} \\
D_{\mathbb{C}^{\ast}}(\widehat{U}, 0) \ar[r]_-{\Upsilon_Y^i} & 
D_{\mathbb{C}^{\ast}}(Y_{\widehat{U}}^+, w^+).
}
\end{align}
Here $\Upsilon_Y^i$ is defined in Theorem~\ref{thm:deq} (iii), 
and $\overline{\Upsilon}_{\widehat{M}}^i$ is defined by
\begin{align}\label{UMi}
\overline{\Upsilon}_{\widehat{M}}^i 
\cneq \otimes \oO_{\widehat{M}^{+}}(i) \circ \dL \pi^{+\ast} \colon 
D_{\mathbb{C}^{\ast}}(\widehat{U}, 0) \to 
D_{\mathbb{C}^{\ast}}(\widehat{M}^{+}, 0). 
\end{align}
\end{lem}
\begin{proof}
An inverse of the equivalence of the right vertical 
arrow in (\ref{dia:commute3}) is given by
$\dR q^+_{\ast} \circ j^{+!}$. 
Therefore it is enough to check 
that 
\begin{align}\label{e:check}
\dR q^+_{\ast} \circ j^{+!} \circ \otimes \oO_{Y_{\widehat{U}}^+}(i) \circ
i_{\widehat{U}\ast}^+ \circ g^{\ast}
\cong \overline{\Upsilon}_{\widehat{M}}^{i+b}[-b].
\end{align}
We use the following commutative 
diagram
\begin{align*}
\xymatrix{
& \ar[ld]_-{\pi^+} 
\widehat{M}^{+} \ar@<-0.3ex>@{^{(}->}[d]_-{\widetilde{j}^+}
\ar@<-0.3ex>@{^{(}->}[r]^-{\widetilde{i}^+} \ar@{}[dr]|\square 
\ar@/^15pt/[rr]^-{\id} 
& 
A^+ \ar@<-0.3ex>@{^{(}->}[d]_-{j^+}
\ar[r]^-{q^+} \ar@{}[dr]|\square
 & \widehat{M}^{+} \ar@<-0.3ex>@{^{(}->}[d]_-{\widetilde{j}^+} 
 \\
\widehat{U} & \ar[l]^-{g} 
\mathbb{P}(V^+)_{\widehat{U}} \ar@<-0.3ex>@{^{(}->}[r]_-{i_{\widehat{U}}^+} &
Y_{\widehat{U}}^+ \ar[r]_-{\mathrm{pr}^+_{\widehat{U}}} 
& \mathbb{P}(V^+)_{\widehat{U}}.
}
\end{align*}
Since the left Cartesian in the above diagram 
is a derived Cartesian, 
by base change 
we have 
$\dL j^{+\ast} \circ i_{\widehat{U}\ast}^+ \cong
\widetilde{i}^+_{\ast}
\circ \dL \widetilde{j}^{+\ast} $.
Together with $\oO_{Y_{\widehat{U}}^+}(1)|_{A^+}=q^{+\ast}\oO_{\widehat{M}^{+}}(1)$, 
we have  
\begin{align*}
&\dR q^+_{\ast} \circ j^{+!} \circ \otimes \oO_{Y_{\widehat{U}}^+}(i) 
\circ i_{\widehat{U}\ast}^+ \circ g^{\ast} \\
&\cong 
 \dR q^+_{\ast} \circ
\otimes \oO_{Y_{\widehat{U}}^+}(b)|_{A^+} \circ 
\dL j^{+\ast} \circ \otimes \oO_{Y_{\widehat{U}}^+}(i) \circ 
i_{\widehat{U}\ast}^+ \circ g^{\ast}
[-b] \\
&\cong 
\otimes \oO_{\widehat{M}^{+}}(b+i) \circ 
 \dR q^+_{\ast} \circ
\dL j^{+\ast}  \circ 
i_{\widehat{U}\ast}^+ \circ g^{\ast}
[-b] \\
&\cong \otimes \oO_{\widehat{M}^{+}}(b+i) \circ \dR q^+_{\ast} \circ \widetilde{i}^+_{\ast}
\circ \dL \widetilde{j}^{+\ast} \circ g^{\ast}[-b] \\
&\cong \otimes \oO_{\widehat{M}^{+}}(b+i) \circ \dL \pi^{+\ast}[-b]
\end{align*}
as expected. 
\end{proof}

By putting all the arguments in this subsection together, 
we have the following: 
\begin{prop}\label{cor:sod}
In the setting of Subsection~\ref{subsec:crit}, 
suppose that the assumptions of Lemma~\ref{lem:Msmooth}
and the dimension condition (\ref{dim:fprod}) hold. 
Then we have the following: 

(i) The functor
\begin{align*}
\Phi_{\widehat{M}} \cneq \Phi^{\oO_{\widehat{M}^{-} \times_{\widehat{U}} \widehat{M}^{+}}} \colon 
D^b(\widehat{M}^{-}) \to D^b(\widehat{M}^{+})
\end{align*}
is fully-faithful. 

(ii) If $n\ge 1$, the functor
\begin{align*}
\Upsilon_{\widehat{M}}^i 
\cneq 
\otimes \oO_{\widehat{M}^+}(i) \circ \dL \pi^{+\ast}
 \colon 
D^b(\widehat{U}) \to D^b(\widehat{M}^+)
\end{align*}
is fully-faithful. 

(iii) We have the SOD
\begin{align*}
D^b(\widehat{M}^{+})=\langle \Imm \Upsilon_{\widehat{M}}^{-n+1}, \ldots, 
\Imm \Upsilon_{\widehat{M}}^0, \Imm \Phi_{\widehat{M}}   \rangle. 
\end{align*}
\end{prop}
\begin{proof}
By Theorem~\ref{thm:deq} (i) and
Proposition~\ref{prop:commute2}, 
the functor
\begin{align*}
\Psi^{\Xi(\oO_{\widehat{M}^{-} \times_{\widehat{U}} \widehat{M}^{+}})} \colon 
D_{\mathbb{C}^{\ast}}(\widehat{M}^{-}, 0) \to 
D_{\mathbb{C}^{\ast}}(\widehat{M}^{+}, 0)
\end{align*}
is fully-faithful. 
Therefore 
(i) follows 
by the commutative diagram (\ref{commute:Xi}). 
Similarly, 
(ii) follows from Theorem~\ref{thm:deq} (ii), 
Lemma~\ref{lem:commute3}
and the commutative diagram (\ref{commute:Xi}). 
As for (iii), 
by Theorem~\ref{thm:deq} (iii), Proposition~\ref{prop:commute2}
and Lemma~\ref{lem:commute3} we have the SOD
\begin{align*}
D^b(\widehat{M}^{+})=
\langle \Imm \Upsilon_{\widehat{M}}^{b-n}, \ldots, 
\Imm \Upsilon_{\widehat{M}}^{b-1}, 
\otimes \oO_{\widehat{M}^{+}}(b-1) \circ \Imm 
\Phi_{\widehat{M}}
 \rangle. 
\end{align*}
By tensoring $\oO_{\widehat{M}^{+}}(1-b)$, we obtain the desired SOD. 
\end{proof}

\section{SOD via d-critical simple flips}\label{sec:dcrit}
In this section, we show that for 
a d-critical simple flip
\begin{align*}
M^+ \to U \leftarrow M^-
\end{align*}
satisfying 
some conditions, 
we have an associated SOD of $D^b(M^+)$.
The SOD in this section is 
obtained by globalizing the 
SOD in Proposition~\ref{cor:sod}. 
\subsection{D-critical simple flips}\label{subsec:mainthm}
Let $U$ be a smooth variety with 
$g \cneq \dim U$. 
Let $(M^{\pm}, s^{\pm})$ be two 
d-critical loci, 
and suppose that we have projective morphisms 
\begin{align}\label{dia:pi}
\xymatrix{
M^+ \ar[rd]_-{\pi^+} & & M^- \ar[ld]^-{\pi^-} \\
& U. &
}
\end{align}
For each $p \in U$, we set
\begin{align}\label{complete:p}
\widehat{U}_p \cneq \Spec \widehat{\oO}_{U, p}, \ 
\widehat{M}_p^{\pm} \cneq M^{\pm} \times_U \widehat{U}_p.
\end{align}

\begin{defi}\label{defi:dsflip}
A diagram (\ref{dia:pi}) is called 
a formal d-critical simple flip if 
for any $p \in U$, 
there exist 
finite dimensional vector spaces $V^{\pm}$ 
with $\dim V^+ \ge \dim V^-$ such that, 
by setting $Y^{\pm}$, $Z$ as in (\ref{def:YZ}), 
and 
\begin{align}\label{formal:YZ}
\widehat{Z}_U \cneq \Spec \widehat{\oO}_{Z_U, (0, p)}, \ 
\widehat{Y}^{\pm}_U \cneq 
Y_U^{\pm} \times_{Z_U} \widehat{Z}_U 
\end{align}
there exist
$\widehat{w} \in \oO_{\widehat{Z}_U}$ and 
a commutative diagram
\begin{align}\label{relchart2}
\xymatrix{
\widehat{M}_p^{\pm} \ar[d]_-{\pi^{\pm}}
 \ar@<-0.3ex>@{^{(}->}[r]^-{\iota^{\pm}} 
& \widehat{Y}_U^{\pm} \ar[d]_-{\widehat{f}_U^{\pm}} 
\ar[rd]^-{\widehat{w}^{\pm}} \\
\widehat{U}_p \ar@<-0.3ex>@{^{(}->}[r]_-{j} & \widehat{Z}_U
 \ar[r]_-{\widehat{w}} & \mathbb{A}^1
}
\end{align}
where horizontal arrows are closed immersions,
$\widehat{w}^{\pm}$ are defined by the above 
commutative diagram, 
$j$ sends $p$ to $(0, p)$ 
and $\iota^{\pm}$ induce the isomorphisms of 
d-critical loci
\begin{align}\label{isom:iota}
\iota^{\pm} \colon 
\widehat{M}_p^{\pm} \stackrel{\cong}{\to}
\{d\widehat{w}^{\pm}=0\} \subset \widehat{Y}_U^{\pm}.
 \end{align}
\end{defi}
For a formal d-critical simple flip (\ref{dia:pi})
and $p \in U$, let $V^{\pm}$ be vector spaces as in 
Definition~\ref{defi:dsflip}. 
Below 
we use the notation in Subsection~\ref{subsec:flip}, 
e.g. $a=\dim V^+$, $b=\dim V^-$, $n \cneq a-b \ge 0$, 
the
coordinates 
$\vec{x}, \vec{y}$ of $V^+$, $V^-$, etc. 
Note that $(a, b)$ may depend on a choice of $p \in U$. 
We assume the following on the diagram (\ref{dia:pi}):

\begin{assum}\label{assum:cond}
(i) The diagram (\ref{dia:pi})
is a formal d-critical simple flip. 

(ii) 
For any $p \in U$, 
the formal function 
$\widehat{w}$ in (\ref{relchart2}) is of the form 
\begin{align}\label{widew}
\widehat{w} 
=\sum_{i, j}x_i y_j w_{ij}^{(1)}(\vec{u})
+\sum_{i, i', j, j'}x_i x_{i'} y_j y_{j'} w_{i i' j j'}^{(2)}(\vec{u}) +\cdots
\end{align}
for some $w_{\ast}^{(k)}(\vec{u}) \in \widehat{\oO}_{U, p}$, and 
$w_{ij}^{(1)}(\vec{u})$ is written as 
\begin{align}\label{wij(1)}
w_{ij}^{(1)}(\vec{u})=\sum_{k=1}^g a_{ijk} u_k +
(\mbox{higher order terms in }\vec{u})
\end{align}
for some  $a_{ijk} \in \mathbb{C}$.
Moreover the bilinear map
\begin{align}\label{psi}
\psi \colon 
\mathbb{C}^a \otimes \mathbb{C}^b \to \mathbb{C}^g, \ 
\psi(\vec{\alpha}, \vec{\beta})=
\left(\sum_{i, j}a_{ijk}\alpha_i \beta_j \right)_{1\le k\le g}
\end{align}
is injective on each factor, i.e 
for any non-zero $\vec{\alpha} \in \mathbb{C}^a$ 
and $\vec{\beta} \in\mathbb{C}^b$, 
the maps
$\psi(\vec{\alpha}, -)$, $\psi(-, \vec{\beta})$ are injective 
maps $\mathbb{C}^b \to \mathbb{C}^g$, 
$\mathbb{C}^a \to \mathbb{C}^g$. 

(iii) There exists a $\pi^{+}$-ample line bundle
$\oO_{M^+}(1)$ on $M^+$ 
such that
under the isomorphism (\ref{isom:iota}), we 
have an isomorphism 
of line bundles 
\begin{align}\label{isom:line}
(\iota^{+})^{\ast}\oO_{\widehat{Y}_U^{+}}(1) \cong \oO_{M^{+}}(1)|_{\widehat{M}_p^+}.
\end{align}
\end{assum}

\begin{lem}\label{lem:smooth}
Suppose that a diagram (\ref{dia:pi})
satisfies Assumption~\ref{assum:cond} (i), (ii). 
Then for any $p \in U$, the critical 
loci $\{d \widehat{w}^{\pm}=0\} \subset \widehat{Y}_U^{\pm}$
in the diagram (\ref{relchart2})
are written as 
\begin{align}\label{write:M2}
&\{d\widehat{w}^+=0\}=\left\{(\vec{x}, \vec{u}) 
\in \mathbb{P}(V^{+})_{\widehat{U}_p} : 
\sum_{i=1}^a x_i w_{ij}^{(1)}(\vec{u})=0
\mbox{ for all }
1\le j\le b\right\}, \\
\notag
&\{d\widehat{w}^-=0\}
=\left\{(\vec{y}, \vec{u}) \in \mathbb{P}(V^{-})_{\widehat{U}_p} : 
\sum_{j=1}^b y_j w_{ij}^{(1)}(\vec{u})=0 \mbox{ for all }
1\le i\le a\right\}.
\end{align}
Moreover 
$n =a-b \ge 0$ is independent of $p \in U$, 
$M^{\pm}$ are smooth 
and satisfy
\begin{align*}
\dim M^{\pm}=\pm n+g-1.
\end{align*}
\end{lem}
\begin{proof}
For $p\in U$, let us consider the diagram 
(\ref{relchart2}). 
The subscheme $\{d\widehat{w}^+=0\} \subset \widehat{Y}_U^+$
is contained in the closed subscheme of $\widehat{Y}_U^+$
defined by the equations
\begin{align}\label{equan}
\frac{\partial \widehat{w}^+(\vec{u})}{\partial u_k}
=\sum_{i, j}x_i y_j 
\frac{\partial w_{ij}^{(1)}(\vec{u})}{\partial u_k}+
\sum_{i, i', j, j'}x_i x_{i'} y_j y_{j'}
\frac{\partial w_{ii'jj'}^{(2)}(\vec{u})}{\partial u_k}
+ \cdots =0
\end{align}
for all $1\le k\le g$. 
Note that we have 
\begin{align*}
\frac{\partial w_{ij}^{(1)}(\vec{u})}{\partial u_k}
=a_{ijk} +O(\vec{u}).
\end{align*}
Then 
by the assumption on the map (\ref{psi}),
the subscheme
\begin{align*}
\left\{ \sum_{i, j}x_i y_j 
\frac{\partial w_{ij}^{(1)}(\vec{u})}{\partial u_k}=0 : 
1\le k\le g  \right\}
\subset (V^{+\ast} \times V^-)_{\widehat{U}_p}
\end{align*} 
coincides with $V^{+\ast} \times \{0\} \times \widehat{U}_p$. 
Since the higher order terms in (\ref{equan})
have degrees bigger than or equal to two in
$\vec{y}$, 
by Nakayama lemma
we see that 
the zero locus defined by the equations (\ref{equan})
equals to 
$\vec{y}=0$ on $\widehat{Y}_{U}^+$, i.e. 
the zero section 
$\mathbb{P}(V^+)_{\widehat{U}_p} \subset \widehat{Y}_{U}^+$. 
Therefore 
$\{d\widehat{w}^+=0\} \subset \widehat{Y}_U^+$ is 
described as (\ref{write:M2}). 

Let $g_j$ for $1\le j\le b$ be the defining 
equations in the RHS of (\ref{write:M2}). 
Again the property on the map (\ref{psi}) implies that the 
Jacobian matrix 
\begin{align*}
\left( \frac{\partial g_j}{\partial x_i}, \frac{\partial g_j}{\partial u_k}
 \right)_{
1\le i\le a, 1\le j\le b, 1\le k\le g}
\end{align*}
is of maximal rank $b$
at any point in the RHS of (\ref{write:M2}). 
Therefore $\{d\widehat{w}^+=0\}$ is smooth of 
dimension $a-1+g-b=n+g-1$. 
By the isomorphism (\ref{isom:iota}), $\widehat{M}_p^+$ is smooth 
of dimension of $n+g-1$ for any $p \in U$, hence 
$M^+$ is smooth of dimension $n+g-1$. 
The claim for $M^-$ is similarly proved. 
\end{proof}

\begin{lem}\label{lem:smooth2}
Under the situation of Lemma~\ref{lem:smooth2}, we have
\begin{align}\label{M:product:g}
\dim (M^- \times_U M^+) \le g-1.
\end{align}
\end{lem}
\begin{proof}
Let us take $p \in U$, 
and vector spaces $V^{\pm}$ as in Definition~\ref{defi:dsflip}
with  
$a=\dim V^+$, $b=\dim V^-$ as before. 
For each $k\ge 0$,
let $U^{(k)} \subset U$ be the locally closed subset 
defined by
\begin{align*}
U^{(k)} \cneq \{x \in U : \dim (\pi^-)^{-1}(x)=k-1\}.
\end{align*}
Then $p \in U^{(b)}$
as $(\pi^-)^{-1}(p)=\mathbb{P}(V^-)$, and 
the descriptions of $\{d\widehat{w}^{\pm}=0\}$ in 
(\ref{write:M2}) 
and the isomorphisms (\ref{isom:iota}) show that 
\begin{align*}
U^{(b)} \cap \widehat{U}_p
=\Spec\left(\widehat{\oO}_{U, p}/(w_{ij}^{(1)}(\vec{u}) : 1\le i \le a, 1\le j\le b) \right).
\end{align*}
It follows that, by the description of $w_{ij}^{(1)}(\vec{u})$ in (\ref{wij(1)}), 
the tangent space of $U^{(b)}$ at $p$ is 
\begin{align*}
TU^{(b)}|_{p}=\left\{  (u_1, \ldots, u_g) \in \mathbb{C}^g : 
\sum_{k=1}^g a_{ijk} u_k=0 : 
1\le i\le a, 1\le j\le b \right\}. 
\end{align*}
Therefore the dimension of $TU^{(b)}|_{p}$ is given by 
the dimension of the cokernel of $\psi$ in (\ref{psi}). 
By the assumption on the map (\ref{psi}), 
the Hopf lemma (see~\cite[Lemma~2]{Ginensky})
implies that 
$\dim \Cok(\psi) \le g-a-b+1$. 
Therefore 
we have
\begin{align}\notag
\dim U^{(b)} \le g-a-b+1.
\end{align}
It follows that
\begin{align*}
&\dim \left( (\pi^+)^{-1}(U^{(b)}) \times_{U^{(b)}} (\pi^-)^{-1}(U^{(b)}) \right) \\
&\le (a-1)+(b-1)+(g-a-b+1) \\
&=g-1.
\end{align*}
Therefore the condition (\ref{M:product:g}) holds. 
\end{proof}

\subsection{SOD under d-critical simple flips}
The following is the main result in this section. 
\begin{thm}\label{thm:main}
Suppose that the diagram (\ref{dia:pi}) satisfies Assumption~\ref{assum:cond}, 
so that $M^{\pm}$ are smooth of dimension 
$\pm n+g-1$ for some $n \in \mathbb{Z}_{\ge 0}$
by Lemma~\ref{lem:smooth}. 
We have the following: 

(i) The functor
\begin{align*}
\Phi_M \cneq \Phi^{\oO_{M^- \times_U M^+}} \colon 
D^b(M^-) \to D^b(M^+)
\end{align*}
is fully-faithful. 

(ii) If $n\ge 1$, the functor
\begin{align*}
\Upsilon_M^i 
\cneq \otimes \oO_{M^+}(i) \circ \dL \pi^{+\ast} \colon 
D^b(U) \to D^b(M^+)
\end{align*}
is fully-faithful. 

(iii) We have the SOD
\begin{align}\label{sod:main}
D^b(M^+)=\langle \Imm \Upsilon_M^{-n+1}, \ldots, 
\Imm \Upsilon_M^0, \Imm \Phi_M   \rangle. 
\end{align}
\end{thm}

We prove Theorem~\ref{thm:main} 
by dividing the proof into three steps. 
\begin{step}
For each $p \in U$, we may assume that the formal function (\ref{widew})
satisfies $w_{\ast}^{(k)}(\vec{u})=0$ for $k\ge 2$. 
\end{step}
\begin{proof}
In the notation of Assumption~\ref{assum:cond} (ii), 
let $w \in \oO_{Z}\otimes \widehat{\oO}_{U, p}$ be 
defined by
\begin{align*}
w=\sum_{i, j} x_i y_j w_{ij}^{(1)}(\vec{u}).
\end{align*}
We set $w^{\pm} \colon Y^{\pm}_{\widehat{U}_p} \to \mathbb{A}^1$
as in the diagram (\ref{def:pm0})
for $\widehat{U}=\widehat{U}_p$. 
Then the argument of Lemma~\ref{lem:smooth}
shows that 
$\{d w^{\pm}=0\} \subset Y_{\widehat{U}_p}^{\pm}$
are described as in the RHS of (\ref{write:M2}) and 
the isomorphisms (\ref{isom:iota}) give
\begin{align*}
\iota^{\pm} \colon 
\widehat{M}_p \stackrel{\cong}{\to}
\{d w^{\pm}=0\} \subset Y_{\widehat{U}_p}^{\pm}.
\end{align*}
Therefore we may replace 
$\widehat{w}$ with $w$ and 
assume that 
$w_{\ast}^{(k)}(\vec{u})=0$ for $k\ge 2$. 
\end{proof}

\begin{step}
Theorem~\ref{thm:main} (i), (ii) hold. 
\end{step}
\begin{proof}
Let $\Phi_{M, R}$ be the right adjoint functor of $\Phi_M$, 
and let $\pP \in D^b(M^- \times M^-)$ be the kernel 
object for the composition functor
\begin{align*}
\Phi_{M, R} \circ \Phi_M \colon D^b(M^-) \to D^b(M^+) \to D^b(M^-).
\end{align*}
Then there is a canonical morphism 
\begin{align}\label{can}
\oO_{\Delta_{M^-}} \to \pP
\end{align}
corresponding to the 
adjunction $\id_{M^-} \to \Phi_{M, R} \circ \Phi_M$. 
Let $\qQ$ be the cone of the morphism (\ref{can}). 
In order to show that $\Phi_M$ is fully-faithful, 
it is enough to show that $\qQ=0$. 
Indeed if this is the case, then 
the adjunction $\id_{M^-} \to \Phi_{M, R} \circ \Phi_M$
is an isomorphism hence $\Phi_M$ is fully-faithful. 
Note that $\qQ$ is supported on the 
fiber product $M^- \times_U M^-$
by the construction. 
Since $\widehat{U}_p \to U$ is faithfully-flat, 
the vanishing $\qQ=0$ is equivalent to 
\begin{align}\label{vanish:Qp}
\qQ \otimes_{\oO_{M^- \times M^-}} \oO_{\widehat{M}_p^-
\times \widehat{M}_p^-}
=0
\end{align}
for all $p \in U$. 

Now by Step~1,
Lemma~\ref{lem:smooth} and Lemma~\ref{lem:smooth2}, 
the diagram
\begin{align*}
\widehat{M}_p^{+} \to \widehat{U}_p \leftarrow
\widehat{M}_p^{-}
\end{align*}
satisfies the assumptions in Proposition~\ref{cor:sod}. 
Then the result of Proposition~\ref{cor:sod} (i) shows that 
the morphism (\ref{can}) is an isomorphism after 
pulling it back by
$\widehat{M}_p^- \times \widehat{M}_p^- \to 
M^- \times M^-$. 
Therefore the vanishing (\ref{vanish:Qp}) holds
for any $p \in U$,  
and Theorem~\ref{thm:main} (i) is proved.  
The proof of (ii) is similar. 
\end{proof}
\begin{step}
Theorem~\ref{thm:main} (iii) holds. 
\end{step}
\begin{proof}
We first show the semiorthogonality of the RHS of (\ref{sod:main}), i.e.
vanishings 
\begin{align*}
\Hom(\Imm \Phi_M, \Imm \Upsilon_M^i)=0, \ 
\Hom(\Imm \Upsilon_M^i, \Imm \Upsilon_M^j)=0,
\end{align*}
for $i<j$.
It is enough to check that 
\begin{align*}
\Phi_{M, R} \circ \Upsilon_M^i=0, \ 
\ \Upsilon_{M, R}^i \circ \Upsilon_M^j=0
\end{align*}
where $\Phi_{M, R}$, $\Upsilon_{M, R}^i$ are the 
right adjoint functors of $\Phi_M$, $\Upsilon_M^i$ respectively. 
Again it is enough to check these vanishings
 formally locally at every $p\in U$, 
and Proposition~\ref{cor:sod} (iii) implies that these 
vanishings hold.

Let $E \in D^b(M^+)$ be an object 
in the right orthogonal complement of the RHS of (\ref{sod:main}). 
Then Proposition~\ref{cor:sod} (iii) 
 implies that $E=0$ on $\widehat{M}_p^+$ for any $p \in U$.
Therefore $E=0$ holds, 
and the RHS of (\ref{sod:main}) generates the LHS. 

\end{proof}

\section{SOD for stable pair moduli spaces}\label{sec:SODpair}
In this section, we apply Theorem~\ref{thm:main} 
to prove Theorem~\ref{intro:thm:PT}, i.e. 
the existence of certain SOD 
on moduli spaces of Pandharipande-Thomas
stable pairs on CY 3-folds. 
\subsection{Stable pairs and stable sheaves}
Let $X$ be a smooth quasi-projective 
variety. By definition, a \textit{stable pair} 
by Pandharipande-Thomas~\cite{PT}
consists of data
\begin{align*}
(F, s), \ s \colon \oO_X \to F
\end{align*}
where $F$ is a pure one dimensional 
coherent sheaf on $X$ with compact support, 
and $s$ is surjective in dimension one. 
For $\beta \in H_2(X, \mathbb{Z})$ and 
$n \in \mathbb{Z}$, 
the moduli space of stable pairs 
$(F, s)$ satisfying the condition
\begin{align}\label{ch:F}
[F]=\beta, \ \chi(F)=n
\end{align}
is
denoted by $P_n(X, \beta)$. 
Here $[F]$ is the homology class of the 
fundamental one cycle associated with $F$. 
The moduli space $P_n(X, \beta)$ is a quasi-projective 
scheme (see~\cite{PT}). 
We define the open subscheme
\begin{align*}
P_n^{\circ}(X, \beta) \subset P_n(X, \beta)
\end{align*}
to be consisting of 
stable pairs $(F, s)$ such that the 
fundamental one cycle associated with $F$ 
is irreducible. 

We denote by $U_n(X, \beta)$
the moduli space of compactly supported 
one dimensional Gieseker stable sheaves $F$ on 
$X$ with respect to a fixed polarization, 
satisfying the condition (\ref{ch:F}).  
The moduli space $U_n(X, \beta)$ is a quasi-projective 
scheme (see~\cite{HL}). We define the 
open subscheme
\begin{align*}
U_n^{\circ}(X, \beta) \subset U_n(X, \beta)
\end{align*}
consisting of one dimensional 
stable sheaves 
whose fundamental one
cycles are irreducible. 
Note that $U_n^{\circ}(X, \beta)$
is the moduli space of pure one dimensional 
sheaves $F$ with irreducible fundamental 
one cycles satisfying (\ref{ch:F}).
In particular, $U_n^{\circ}(X, \beta)$ is independent of a 
choice of a polarization. 
\begin{rmk} 
Alternatively, $U_n^{\circ}(X, \beta)$
parametrizes pairs 
$(C, F)$
where $C \subset X$ is an irreducible projective curve
with $[F]=\beta$, 
and $F \in \Coh(C)$ is a rank one torsion free sheaf
satisfying $\chi(F)=n$. 
\end{rmk}

\subsection{Wall-crossing diagram of stable pair moduli spaces}
\label{subsec:wallpair}
Suppose that $X$ is a smooth projective 
CY 3-fold, i.e. 
\begin{align*}
\dim X=3, \ K_X=0.
\end{align*} 
Let us take $\beta \in H_2(X, \mathbb{Z})$
and $n \in \mathbb{Z}_{\ge 0}$. 
Then as in~\cite{PT3}, we have the diagram
\begin{align}\label{dia:PT}
\xymatrix{
P_n^{\circ}(X, \beta) \ar[rd]_-{\pi^+} & & \ar[ld]^-{\pi^-} 
P_{-n}^{\circ}(X, \beta) \\
& U_n^{\circ}(X, \beta). &
}
\end{align}
Here $\pi^{\pm}$ are defined by
\begin{align*}
\pi^+(F, s)=F, \ 
\pi^-(F', s')=\eE xt^2_X(F', \oO_X). 
\end{align*}
If furthermore $H^1(\oO_X)=0$, 
then the diagram (\ref{dia:PT}) 
gives an example of an
analytic (in particular formal) 
d-critical simple flip 
(see~\cite[Theorem~6.18]{Toddbir}). 
Here we recall some more details. 

Let us take a point $p \in U_n^{\circ}(X, \beta)$
corresponding to a pure one dimensional sheaf $F$
on $X$. We write
\begin{align*}
&\widehat{U}_n(X, \beta)_p \cneq 
\Spec 
\widehat{\oO}_{U_n(X, \beta), p}, \\
&\widehat{P}_n(X, \beta)_p \cneq 
P_n^{\circ}(X, \beta) \times_{U_n^{\circ}(X, \beta)}
\widehat{U}_n(X, \beta)_p. 
\end{align*}
 We take a collection of objects in
$D^b(X)$
\begin{align}\label{E12}
E_{\bullet}=(E_1, E_2), 
\ E_1=\oO_X, \ 
E_2=F[-1].
\end{align}
We set vector spaces $V^+$, $V^-$ and $U$ as follows:
\begin{align}
&\label{V:Ext}V^+ \cneq \Ext^1_X(E_1, E_2)=H^0(F), \\ 
&\notag V^- \cneq \Ext^1_X(E_2, E_1)=H^1(F)^{\vee}, \\
&\notag U \cneq \Ext^1_X(E_2, E_2)=\Ext_X^1(F, F). 
\end{align}
Below we use the notation and 
convention in Subsection~\ref{subsec:flip}
and Subsection~\ref{subsec:crit}, e.g. 
$\mathbb{C}^{\ast}$-actions on $V^{\pm}$, 
the GIT quotients $Y^{\pm}$, $Z$, 
coordinates $\vec{x}$, $\vec{y}$, $\vec{u}$
on $V^{\pm}$, $U$, 
$a=\dim V^+$, $b=\dim V^-$, $g=\dim U$, 
etc. 
We also take the formal completion
$\widehat{Z}_U$ of $Z_U$ at $(0, 0)$, 
and set $\widehat{f}_U^{\pm} \colon \widehat{Y}_U^{\pm} \to \widehat{Z}_U$
as in (\ref{formal:YZ}). 
The following result is obtained in~\cite{Toddbir}: 
\begin{thm}\emph{(\cite[Theorem~6.18]{Toddbir})}\label{thm:Pflip}
In the above situation, 
there exist an element
$\widehat{w} \in \widehat{\oO}_{Z_U, (0, 0)}$ 
and the
commutative diagram

\begin{align}\label{dia:YZ}
\xymatrix{
\widehat{P}_n(X, \beta)_p \ar[r]^-{\cong}_-{\iota^{\pm}} \ar[d]_-{\pi^+}
& \{d\widehat{w}^{\pm}=0\} \ar@<-0.3ex>@{^{(}->}[r] \ar[d]
& \widehat{Y}_U^{\pm} 
\ar[d]_{\widehat{f}_U^{\pm}} \ar[rd]^-{\widehat{w}^{\pm}} &\\
\widehat{U}_n(X, \beta)_p \ar[r]^-{\cong} &
\{dw^{(0)}=0\} \ar@<-0.3ex>@{^{(}->}[r]_-{j}
& \widehat{Z}_U \ar[r]_-{\widehat{w}}  & \mathbb{A}^1.
}
\end{align}
Here $\widehat{w}^{\pm}$
are defined by the above commutative diagram, 
the bottom left arrow 
sends $p$ to $(0, 0)$, 
the map $j$
is 
the composition of the inclusion 
$\{dw^{(0)}=0\} \subset \widehat{U}$ with 
the inclusion 
$\widehat{U} \hookrightarrow \widehat{Z}_U$
given by 
$u\mapsto (0, u)$. 
\end{thm}
\begin{rmk}
In~\cite[Theorem~6.18]{Todbir}, 
it is stated that we can take $\widehat{w}$ as an 
analytic function on an analytic open neighborhood 
of $0 \in Z_U$, 
and the diagram (\ref{dia:YZ})
can be extended to analytic 
neighborhoods of $0 \in Z_U$ and 
$p \in U_n^{\circ}(X, \beta)$. 
The 
formal version in Theorem~\ref{thm:Pflip} 
is weaker than the analytic version in~\cite[Theorem~6.18]{Todbir}, 
but enough for the purpose of this paper.  
\end{rmk}

Let us write the formal function $\widehat{w}$ 
in Theorem~\ref{thm:Pflip} as
\begin{align}\label{formal:w}
\widehat{w} 
=w^{(0)}(\vec{u})+\sum_{i, j}x_i y_j w_{ij}^{(1)}(\vec{u})
+\sum_{i, i', j, j'}x_i x_{i'} y_j y_{j'} w_{i i' j j'}^{(2)}(\vec{u}) +\cdots
\end{align}
for $w_{\ast}^{(k)}(\vec{u}) \in \widehat{\oO}_{U, 0}$. 
The formal function (\ref{formal:w})
is 
constructed using the 
minimal cyclic $A_{\infty}$-structure on 
the subcategory of $D^b(X)$ generated by 
$E_1$ and $E_2$
(see~\cite[Subsection~5.1]{Toddbir}). 
In particular, the linear term of $w_{ij}^{(1)}(\vec{u})$ is 
give as follows. 
Let us consider the triple product 
\begin{align}\notag
\Ext_X^1(E_2, E_2) \otimes \Ext_X^1(E_1, E_2) \otimes
&\Ext_X^1(E_2, E_1) \\
\label{triple}
&\to \Ext_X^3(E_2, E_2) 
\cong \mathbb{C}
\end{align}
given by the composition, 
where the last isomorphism is given by the 
Serre duality. 
For $1\le i\le a$, $1\le j\le b$
and $1\le k \le g$, let 
\begin{align*}
x_i^{\vee} \in \Ext_X^1(E_1, E_2), \ 
y_j^{\vee} \in \Ext_X^1(E_2, E_1), \ 
u_k^{\vee} \in \Ext_X^1(E_2, E_2)
\end{align*}
be the dual basis of 
$x_i$, $y_j$, $u_k$
respectively. 
Then using the triple product (\ref{triple}), 
we have
\begin{align}\label{w:linear}
w_{ij}^{(1)}(\vec{u})=
\frac{1}{2} \sum_{k=1}^g(x_i^{\vee} \cdot y_j^{\vee} \cdot u_k^{\vee}) 
u_k +(\mbox{higher order terms in } \vec{u}). 
\end{align}
We show that 
$w_{ij}^{(1)}(\vec{u})$ satisfies the condition 
in Assumption~\ref{assum:cond} (ii): 
\begin{lem}\label{lem:inj}
The map 
\begin{align}\label{ext:compose}
\Ext_X^1(E_1, E_2) \otimes \Ext_X^1(E_2, E_1) \to \Ext_X^2(E_2, E_2)
\end{align}
given by the composition 
is injective on each factors. 
\end{lem}
\begin{proof}
Recall that 
 $E_1$, $E_2$ are taken as in (\ref{E12}), 
i.e. $E_1=\oO_X$ and $E_2=F[-1]$ for a pure one dimensional 
sheaf $F$ on $X$ 
with irreducible fundamental one cycle. 
 So $F$ is written as $j_{\ast}E$ 
where $j \colon C \hookrightarrow X$ is an irreducible
 Cohen-Macaulay curve
and $E$ is a rank one torsion free sheaf on $C$. 
Therefore the map (\ref{ext:compose}) is 
\begin{align}\label{compose:ext}
H^0(C, E) \otimes \Ext_X^2(j_{\ast}E, \oO_X)
\to \Ext_X^2(j_{\ast}E, j_{\ast}E). 
\end{align}
Note that 
\begin{align*}
\Ext_X^2(j_{\ast}E, \oO_X)=
\Ext_C^2(E, j^{!}\oO_X)=
\Hom(E, \omega_C)
\end{align*}
where $\omega_C$ is the dualizing sheaf on $C$. 
Also we have 
$H^1(C, \eE nd(E)) \subset \Ext_X^1(j_{\ast}E, j_{\ast}E)$, 
and the Serre duality gives the surjection
\begin{align*}
\Ext_X^2(j_{\ast}E, j_{\ast}E) \twoheadrightarrow
\Hom(\eE nd(E), \omega_C). 
\end{align*}
By composing it with (\ref{compose:ext}) we obtain the map
\begin{align}\label{inj:omega}
H^0(C, E) \otimes \Hom(E, \omega_C) \to
\Hom(\eE nd(E), \omega_C). 
\end{align}
The above bilinear map is 
given by the natural composition map. 
Since $E$, $\eE nd(E)$ are torsion free on $C$, 
and $\omega_C$ is also torsion free on $C$ as 
$C$ is Cohen-Macaulay, the 
bilinear map (\ref{inj:omega}) is injective on each factors. 
Therefore the lemma holds. 
\end{proof}

We also have the following lemma:  
\begin{lem}\label{lem:O1}
There is a $\pi^+$-ample line bundle 
$\oO_P(1)$ on $P_n^{\circ}(X, \beta)$ such that
for any $p \in U_n^{\circ}(X, \beta)$, 
the isomorphisms $\iota^{+}$ in the diagram 
(\ref{thm:Pflip})
satisfies
\begin{align*}
(\iota^{+})^{\ast}(\oO_{\widehat{Y}_U^{+}}(1)|_{\{d\widehat{w}^{+}=0\}})
\cong \oO_{P}(1)|_{\widehat{P}_n(X, \beta)_p}. 
\end{align*}
\end{lem}
\begin{proof}
Let $H$ be a sufficiently ample divisor on $X$ such that 
for any $[F] \in U_n^{\circ}(X, \beta)$, 
the sheaf $F(H) \cneq F \otimes \oO_X(H)$
satisfies  
$H^1(X, F(H))=0$ 
and the natural map 
$F \to F(H)$ is injective. 
Such an ample divisor $H$ exists as $U_n^{\circ}(X, \beta)$ is of finite type. 
By setting $d=H \cdot \beta$, we have the commutative diagram
\begin{align}\label{dia:stable}
\xymatrix{
P_n^{\circ}(X, \beta) 
 \ar@<-0.3ex>@{^{(}->}[r] \ar[d]_-{\pi^+}
& P_{n+d}^{\circ}(X, \beta)
\ar[d]^-{\pi^+} \\
U_n^{\circ}(X, \beta) \ar[r]^-{\cong}
& U_{n+d}^{\circ}(X, \beta).
}
\end{align}
Here the top arrow is given by 
\begin{align*}
(\oO_X \to F)
\mapsto 
(\oO_X \to F \hookrightarrow F(H))
\end{align*}
and the bottom arrow 
sends a stable sheaf $F$ to $F(H)$. 
By the condition $H^1(X, F(H))=0$, the right
arrow is a projective bundle with fiber $\mathbb{P}(H^0(X, F(H)))$. 
By restricting the tautological 
line bundle on $P_{n+d}^{\circ}(X, \beta)$
to
$P_n^{\circ}(X, \beta)$ by the 
top arrow of (\ref{dia:stable}), we obtain the 
desired $\oO_P(1)$. 
\end{proof}

\subsection{SOD for stable pair moduli spaces}
We keep the situation in the previous 
subsections. For the diagram (\ref{dia:PT}), let 
$\wW^{\circ}$ be the fiber product
\begin{align*}
\wW^{\circ} \cneq P_n^{\circ}(X, \beta)
 \times_{U_n^{\circ}(X, \beta)}P_{-n}^{\circ}(X, \beta).
\end{align*}
The following is the main result in this section:
\begin{thm}\label{thm:PTpair}
For $n\ge 0$ and $\beta \in H_2(X, \mathbb{Z})$, 
suppose that $U_n^{\circ}(X, \beta)$
is non-singular of dimension $g$. 
Then $P_{\pm n}^{\circ}(X, \beta)$ are also 
non-singular with dimension $\pm n+g-1$, 
and we have the following: 

(i) The functor
\begin{align*}
\Phi_P \cneq 
\Phi^{\oO_{\wW^{\circ}}} 
 \colon D^b(P_{-n}^{\circ}(X, \beta)) \to D^b(P_n^{\circ}(X, \beta))
\end{align*}
is fully-faithful. 

(ii) There is a $\pi^{+}$-ample line bundle 
$\oO_P(1)$ on $P_n^{\circ}(X, \beta)$ such that
if $n\ge 1$, the functor
\begin{align*}
\Upsilon^i_P \colon D^b(U_n^{\circ}(X, \beta)) \to D^b(P_n^{\circ}(X, \beta))
\end{align*}
defined by $\dL\mathrm{\pi^+}^{\ast}(-) \otimes \oO_P(i)$ is
 fully-faithful. 

(iii) We have the semiorthogonal decomposition
\begin{align*}
D^b(P_n^{\circ}(X, \beta))=
\langle \Imm \Upsilon_P^{-n+1}, \ldots, \Imm \Upsilon_P^{0}, \Imm \Phi_P
 \rangle.
\end{align*}
\end{thm}
\begin{proof}
We show that the diagram (\ref{dia:PT}) satisfies 
Assumption~\ref{assum:cond}. 
Let us take $p \in U_n^{\circ}(X, \beta)$
corresponding to a pure one dimensional sheaf $F$. 
The assumption that $U_n^{\circ}(X, \beta)$ is smooth and the
bottom left isomorphism in the diagram 
(\ref{dia:YZ}) indicate that, for 
the formal function $\widehat{w}$
written as (\ref{formal:w}), we 
may assume that $w^{(0)}(\vec{u})=0$. 
Then
Assumption~\ref{assum:cond} (i) follows 
from Theorem~\ref{thm:Pflip}, (ii) 
follows from Lemma~\ref{lem:inj} and (iii)
follows from Lemma~\ref{lem:O1}. 
Therefore theorem follows from Theorem~\ref{thm:main}. 
\end{proof}

\begin{rmk}\label{rmk:CY}
When $X$ is a non-compact CY 3-fold, 
suppose that $X$ has a smooth compactification 
$X \subset \overline{X}$
such that $H^i(\oO_{\overline{X}})=0$ for $i=1, 2$. 
Then the result of Theorem~\ref{thm:PTpair} also holds 
in this case 
without any modification, by
replacing $X$ with $\overline{X}$. 
This is because for $E_1=\oO_{\overline{X}}$ and 
$E_2=F[-1]$ where the support of $F$ is contained in $X$, 
we have the perfect pairing 
\begin{align*}
\Ext_X^1(E_i, E_j) \otimes \Ext_X^2(E_j, E_i) \to \mathbb{C}
\end{align*}
by the CY3 condition of $X$ and 
the vanishing $H^i(\oO_{\overline{X}})=0$ for $i=1, 2$.
\end{rmk}

\subsection{Stable pairs on local surfaces}
We apply Theorem~\ref{thm:PTpair} to some local surfaces. 
Let $S$ be a smooth projective surface 
satisfying $H^i(\oO_S)=0$ for $i=1, 2$. 
We consider the non-compact CY 3-fold
\begin{align*}
X =\mathrm{Tot}_S(K_S).
\end{align*}
We will apply Theorem~\ref{thm:PTpair}
to show the existence of 
SOD of relative Hilbert schemes of points on
the universal curve over a complete linear system.  

Let us take $\beta \in H_2(S, \mathbb{Z})=H^2(S, \mathbb{Z})$
such that $-K_S \cdot \beta>0$. 
By the assumption $H^i(\oO_S)=0$ for $i=1, 2$, there is unique 
$L \in \Pic(S)$ such that 
$c_1(L)=\beta$. 
Let
$\lvert L \rvert^{\circ} \subset \lvert L \rvert$
be the open subset consisting of 
irreducible curves and 
\begin{align*}
\pi \colon \cC \to \lvert L \rvert^{\circ}
\end{align*}
the universal curve. 
Note that any member $C \in \lvert L \vert^{\circ}$ has 
the following arithmetic genus
\begin{align*}
g=1+\frac{1}{2}(\beta^2+K_S \cdot \beta).
\end{align*}
We have the following diagram
\begin{align*}
\xymatrix{
\cC^{[n]} \ar[rd]_-{\pi^{[n]}} & & J_n \ar[ld]^-{\pi_J} \\
& \lvert L \rvert^{\circ}. 
}
\end{align*}
Here $\pi^{[n]}$ is the $\pi$-relative Hilbert scheme of $n$-points, 
and $\pi_J$ is the $\pi$-relative rank one torsion free sheaves
on the fibers of $\pi$ with Euler characteristic $n$. 
Let $i \colon S \hookrightarrow X$ be the zero section.
We have the following lemma: 
\begin{lem}\label{lem:locsurface}
(i) We have isomorphisms
\begin{align*}
\cC^{[n+g-1]} \stackrel{\cong}{\to} P_n^{\circ}(S, \beta) \stackrel{\cong}{\to} P_n^{\circ}(X, i_{\ast}\beta)
\end{align*}
and they are non-singular. 

(ii) We have isomorphisms
\begin{align*}
J_n \stackrel{\cong}{\to} U_n^{\circ}(S, \beta) \stackrel{\cong}{\to}
U_n^{\circ}(X, i_{\ast}\beta)
\end{align*}
and they are non-singular. 
\end{lem}
\begin{proof}
As for (i), the isomorphism 
$\cC^{[n+g-1]} \stackrel{\cong}{\to} 
P_n^{\circ}(S, \beta)$
and the smoothness of 
$P_n^{\circ}(S, \beta)$ follow from the argument 
of~\cite[Proposition~B.8, Proposition~C.2]{PT3}.
The assumption $-K_S \cdot \beta>0$ implies that 
any compactly supported irreducible curve on $X$ with 
homology class $i_{\ast}\beta$ must lie on 
the zero section $S \subset X$. 
Therefore we have 
the set theoretic bijection
$P_n^{\circ}(S, \beta) \to
P_n^{\circ}(X, i_{\ast}\beta)$, and they 
have the same scheme structures 
by~\cite[Proposition~3.4]{MR3238154}. 

As for (ii), the smoothness 
of $U_n^{\circ}(S, \beta)$ follows from 
\begin{align*}
\Ext_S^2(F, F)=\Hom(F, F \otimes \oO_S(K_S))^{\vee}=0
\end{align*}
for a sheaf $F$ corresponding to a point in $U_n(X, \beta)$,
by the Serre duality and the assumption 
$-K_S \cdot \beta>0$. 
The isomorphism 
$J_n \stackrel{\cong}{\to} U_n^{\circ}(S, \beta)$
follows 
from the argument in~\cite[Subsection~5.3]{MT}, and the 
isomorphism 
$U_n^{\circ}(S, \beta) \stackrel{\cong}{\to}
U_n^{\circ}(X, i_{\ast}\beta)$
follows similarly to (i). 
\end{proof}
By Lemma~\ref{lem:locsurface},
the diagram (\ref{dia:PT}) in this case 
is
\begin{align}\label{dia:PT:loc}
\xymatrix{
\cC^{[n+g-1]} \ar[rd]_-{\pi^+} & & \ar[ld]^-{\pi^-} \cC^{[-n+g-1]} \\
& J_n. &
}
\end{align}
Applying Theorem~\ref{thm:PTpair}
to $X=\mathrm{Tot}_S(K_S)$ 
and noting Remark~\ref{rmk:CY}, 
we have the following: 
\begin{cor}\label{cor:locsurface}
For each $n\ge 0$, we have the SOD
\begin{align*}
D^b(\cC^{[n+g-1]})=
\langle \overbrace{D^b(J_{n}), \ldots, 
D^b(J_{n})}^{n}, D^b(\cC^{[-n+g-1]}) \rangle. 
\end{align*}
\end{cor}

\subsection{SOD of symmetric product of curves}
Let $C$ be a smooth projective curve
over $\mathbb{C}$ of genus $g$. Its $k$-fold 
symmetric product $C^{[k]}$ is defined by 
\begin{align*}
C^{[k]} \cneq \overbrace{(C \times \cdots \times C)}^k/\mathfrak{S}_k
\end{align*}
where the action of the symmetric group $\mathfrak{S}_k$ is given by 
the permutation. 
The variety $C^{[k]}$ is a smooth projective variety of dimension $k$, and 
identified with the Hilbert scheme of $k$-points on $C$. 

Let $\pP ic^k(C)$ be the moduli space of degree $k$ line bundles on $C$, which is a $g$-dimensional complex torus. 
Once we fix a point $c \in C$, 
we have the isomorphism
\begin{align}\label{isom:pic}
\pP ic^k(C) \stackrel{\cong}{\to} J_C \cneq \pP ic^0(C)
\end{align}
which sends $[L] \in \pP ic^k(C)$ to 
$[L(-kc)] \in J_C$. 
Below we fix the above isomorphisms for each 
$k \in \mathbb{Z}$. 
We also have the Abel-Jacobi map 
\begin{align}\label{intro:AJ}
\mathrm{AJ} \colon 
C^{[k]} \to \pP ic^k(C)
\end{align}
which sends a length $k$ subscheme 
$Z \subset C$
to the line bundle $\oO_C(Z)$. 
\begin{rmk}\label{rmk:AJ}
For $k>2g-2$, the map (\ref{intro:AJ})
is a projective bundle. In general, 
the map (\ref{intro:AJ}) is a stratified projective bundle, 
where stratas on $\Pic^k(C)$ are given by 
Brill-Noether loci. 
The geometry of Brill-Noether loci is complicated and depends on 
the complex structure of $C$, whose study is a classical subject 
on the study of symmetric products of curves 
(see~\cite[Section~5]{Flamini}, \cite[Example~1.0.7-1.0.10]{MR3114949}). 
\end{rmk}

For $n\ge 0$, we consider the following 
diagram
\begin{align}\label{diagram:sym}
\xymatrix{
C^{[n+g-1]} \ar[rd]_-{\mathrm{AJ}} &  & 
C^{[-n+g-1]} \ar[ld]^-{\mathrm{AJ}^{\vee}} \\
&  \pP ic^{n+g-1}(C). &
}
\end{align}
Here $\mathrm{AJ}^{\vee}$ 
sends $Z \subset C$ to $\omega_C(-Z)$. 
Applying Theorem~\ref{thm:PTpair}
and using the isomorphism (\ref{isom:pic}), 
we obtain the following corollary:

\begin{cor}\label{thm:sod:sym}
For each $n\ge 0$, we have the SOD
\begin{align*}
D^b(C^{[n+g-1]})=
\langle \overbrace{D^b(J_C), \ldots, 
D^b(J_C)}^{n}, D^b(C^{[-n+g-1]}) \rangle. 
\end{align*}
\end{cor} 
\begin{proof}
Let 
$X$ be the non-compact CY 3-fold
\begin{align*}
X=\mathrm{Tot}_C(L_1 \oplus L_2)
\end{align*}
where $L_1$, $L_2$ are general line bundles 
of degree $g-1$ satisfying $L_1 \otimes L_2 \cong \omega_C$.
Then the diagram (\ref{dia:PT}) in this 
case coincides with the diagram (\ref{diagram:sym}).
As mentioned in~\cite[Example~9.22, Remark~9.23]{Toddbir}, 
the result of Theorem~\ref{thm:Pflip} applies 
to the non-compact CY 3-fold $X$. 
Therefore the result 
follows by the argument of Theorem~\ref{thm:PTpair}
and isomorphisms (\ref{isom:pic}). 
\end{proof}

For $n=0$, the images of
$\mathrm{AJ}$ and $\mathrm{AJ}^{\vee}$ coincide with
the theta divisor
\begin{align*}
\Theta \cneq \{ [L] \in \pP ic^{g-1}(C) : 
h^0(L) \neq 0\} \subset \pP ic^{g-1}(C)
\end{align*}
which is singular in general, but 
has only rational singularities~\cite{MR0349687}. 
So we have the diagram
\begin{align*}
\xymatrix{
C^{[g-1]} \ar[rd]_-{\mathrm{AJ}}  \ar@{.>}[rr]
& & C^{[g-1]} \ar[ld]^-{\mathrm{AJ}^{\vee}} \\
& \Theta &
}
\end{align*}
which gives a (possibly non-isomorphic) 
resolutions of $\Theta$. 
Let $\wW$ be the fiber 
product of the above diagram. 
Applying Theorem~\ref{thm:PTpair}
as in the proof of Corollary~\ref{thm:sod:sym} for $n=0$, 
we have the following:  
\begin{cor}\label{cor:symequiv}
We have the autequivalence
\begin{align}\label{cor:S}
\Phi^{\oO_{\wW}} \colon 
D^b(C^{[g-1]}) \stackrel{\sim}{\to}
D^b(C^{[g-1]}).
\end{align}
\end{cor}

Below we give some examples on Corollary~\ref{thm:sod:sym}
and Corollary~\ref{cor:symequiv}. 
\begin{exam} Suppose that $n>g-1$. Then 
$C^{[-n+g-1]}=\emptyset$ and 
\begin{align*}
\mathrm{AJ} \colon C^{[n+g-1]} \to \pP ic^{n+g-1}(C)
\end{align*}
is a projective bundle whose fibers are
$\mathbb{P}^{n-1}$. 
Then the 
SOD in Theorem~\ref{thm:sod:sym} is
\begin{align*}
D^b(C^{[n+g-1]})=
\langle \overbrace{D^b(J_C), \ldots, 
D^b(J_C)}^{n}
 \rangle
\end{align*}
which is nothing but Orlov's SOD 
for projective bundles~\cite{MR1208153}. 
\end{exam}

\begin{exam}
Suppose that $n=g-1$. 
Then $C^{[-n+g-1]}=\Spec \mathbb{C}$
and 
\begin{align*}
\mathrm{AJ} \colon C^{[2g-2]} \to \pP ic^{2g-2}(C)
\end{align*}
is a projective bundle 
outside the point $[\omega_C] \in \pP ic^{2g-2}(C)$. 
For the fiber $F=\mathrm{AJ}^{-1}([\omega_C])$, 
its structure sheaf $\oO_F$ is exceptional, and the 
SOD in Theorem~\ref{thm:sod:sym}
is
\begin{align*}
D^b(C^{[2g-2]})=
\langle \overbrace{D^b(J_C), \ldots, 
D^b(J_C)}^{g-1}, \oO_F \rangle. 
\end{align*}
\end{exam}

\begin{exam}
 Suppose that $n=g-2$. 
Then $C^{[-n+g-1]}=C$ and 
\begin{align*}
\mathrm{AJ} \colon C^{[2g-3]} \to \pP ic^{2g-3}(C)
\end{align*}
is a projective bundle outside 
$\mathrm{AJ}^{\vee}(C) \subset \pP ic^{2g-3}(C)$. 
In this case, the SOD in Theorem~\ref{thm:sod:sym} is
\begin{align*}
D^b(C^{[2g-3]})=
\langle \overbrace{D^b(J_C), \ldots, 
D^b(J_C)}^{g-2}, D^b(C) \rangle. 
\end{align*}
\end{exam}

\begin{exam}
 Suppose that $g=3$ and $n=0$. 
Then the birational map 
\begin{align*}
\mathrm{AJ} \colon C^{[2]} \to \Theta
\end{align*}
is not an isomorphism if and only if 
$C$ is a hyper-elliptic curve (see~\cite[Example~1.0.9]{MR3114949}). 
In this case, the above map contracts a 
$(-2)$-curve on $C^{[2]}$
to a rational double point in $\Theta$. 
The equivalence (\ref{cor:S})
is the spherical twist along with the $(-2)$-curve. 
\end{exam}

\begin{exam}
 Suppose that $g=4$ and $n=1$. Then the birational map
\begin{align*}
\mathrm{AJ} \colon C^{[4]} \to \pP ic^4(C)
\end{align*}
contracts a divisor $E \subset C^{[4]}$
to the surface $\mathrm{AJ}^{\vee}(C^{[2]}) \subset \pP ic^4(C)$
(see~\cite[Example~1.0.10]{MR3114949}). 
If $C$ is not hyperelliptic, then $E$ is 
a $\mathbb{P}^1$-bundle over 
$\mathrm{AJ}^{\vee}(C^{[2]}) \cong C^{[2]}$. 
The SOD in Theorem~\ref{thm:sod:sym} becomes
\begin{align*}
D^b(C^{[4]})=
\langle D^b(J_C), D^b(C^{[2]})\rangle.
\end{align*}
If $C$ is not hyperelliptic, 
the above SOD seems to be the blow-up formula
of derived categories obtained in~\cite{MR1208153}. 
\end{exam}

\begin{exam}
 Suppose that $g=4$ and $n=0$. 
Then the birational map 
\begin{align*}
\mathrm{AJ} \colon C^{[3]} \to \Theta
\end{align*}
is a crepant resolution of $\Theta$ which is 
a divisorial contraction if $C$ is hyperelliptic, 
small resolution which contracts 
one or two smooth rational curves
if $C$ is not hyperelliptic
(see~\cite[Example~1.0.10]{MR3114949}). 
In the latter case, the 
equivalence (\ref{cor:S})
seems to be the 
derived equivalence under flops~\cite{B-O1, Br1}. 
\end{exam}

\section{Categorification of Kawai-Yoshioka formula}\label{sec:KY}
In this section, we prove Theorem~\ref{intro:thm2} as 
another application 
of Theorem~\ref{thm:main}.
We use Kawai-Yoshioka's diagram~\cite{KY} 
relating moduli spaces of stable pairs on K3 surfaces
with moduli spaces of stable sheaves on them. 
The key ingredient, 
which was essentially observed in~\cite{TodK3}, is to interpret 
Kawai-Yoshioka's diagram in terms of wall-crossing diagram in a
CY 3-fold defined by 
the product of the K3 surface and an elliptic curve.
\subsection{SOD of relative Hilbert schemes of points}
Let $S$ be a smooth projective K3 surface such that
\begin{align*}
\Pic(S)=\mathbb{Z}[\oO_S(H)]
\end{align*}
for an ample divisor $H$ on $S$. 
Let $g \in \mathbb{Z}$ be 
defined by $H^2=2g-2$. 
We have the complete linear system
$\vert H \rvert$ and the universal curve
\begin{align*}
\pi \colon 
\cC \to 
\lvert H \rvert=\mathbb{P}^g.
\end{align*}
In what follows, we fix 
$n \ge 0$. Let 
\begin{align}\label{rel:hilb}
\cC^{[n+g-1]}
\to \mathbb{P}^g
\end{align}
be the $\pi$-relative Hilbert scheme of $(n+g-1)$-points on $\cC$. 
As in Lemma~\ref{lem:locsurface}, 
the $\pi$-relative Hilbert scheme (\ref{rel:hilb})
is isomorphic to the 
the moduli space of Pandharipande-Thomas stable 
pair moduli space $P_n(S, [H])$ on $S$.

Let $\Gamma_S$ be the \textit{Mukai lattice} of $S$
\begin{align*}
\Gamma_S \cneq H^0(S, \mathbb{Z}) \oplus \mathbb{Z}[H] \oplus 
H^4(S, \mathbb{Z}). 
\end{align*}
For $E \in D^b(S)$ 
its \textit{Mukai vector} is defined by
\begin{align*}
v(E) \cneq \ch(E) \cdot \sqrt{\td_S}
 \in \Gamma_S.
\end{align*}
For elements 
$(r_i, \beta_i, m_i) \in \Gamma_S$ with $i=1, 2$, 
the \textit{Mukai pairing} is defined by
\begin{align*}
((r_1, \beta_1, m_1), (r_2, \beta_2, m_2))
 \cneq \beta_1 \beta_2-r_2 m_1 -r_1 m_2.
\end{align*}
For each $k \in \mathbb{Z}_{\ge 0}$, we 
define $U_k$ to be 
the moduli space of $H$-Gieseker stable sheaves
$E$ on $S$
satisfying
\begin{align*}
v(E) =\mathbf{v}_k \cneq (k, [H], k+n). 
\end{align*}
Here we refer to~\cite{HL} for basics on 
moduli spaces of stable sheaves and 
their properties. 
The moduli space $U_k$ is known to be 
a projective 
irreducible holomorphic symplectic manifold of dimension
given by
\begin{align}\notag
\dim U_k &=2+(\mathbf{v}_k, \mathbf{v}_k) \\
\label{dim:Mk}
&= 2(g-k^2-kn).
\end{align}
Let $\pP_k$ be the moduli space of pairs
\begin{align*}
(E, s), \ s \colon \oO_S \to E
\end{align*}
where $[E] \in U_k$ and 
$s$ is a non-zero morphism. 
By~\cite[Lemma~5.117]{KY}, the 
moduli space $\pP_k$ is a smooth projective variety
whose dimension is 
\begin{align}\notag
\dim \pP_k &= 1+\langle \mathbf{v}_{k-1}, \mathbf{v}_k \rangle \\
\label{dim:Pk}
&=2(g-k^2-kn)+2k+n-1.
\end{align}
We have the following diagram (see~\cite[Lemma~5.113]{KY})
\begin{align}\label{dia:P+}
\xymatrix{
\pP_k  \ar[rd]_-{\pi_k^+} & & \ar[ld]^-{\pi_{k}^-} \pP_{k+1} \\
& U_k. &
}
\end{align}
Here $\pi_k^{\pm}$ are defined by
\begin{align*}
\pi_k^+(E, s) \cneq E, \ 
\pi_k^-(E', s') \cneq \Cok(s').
\end{align*}
As an application of Theorem~\ref{thm:main}, 
we have the following result 
whose proof will be given in 
Subsection~\ref{subsec:proofK3}:
\begin{thm}\label{thm:K3}
For $k\ge 0$, we have the following SOD
\begin{align*}
D^b(\pP_k)=\langle \overbrace{D^b(U_k), \ldots, D^b(U_k)}^{n+2k}, D^b(\pP_{k+1}) \rangle.
\end{align*}
\end{thm}

Let $N\ge 0$ be defined by
\begin{align*}
N \cneq \mathrm{max} \{ k \ge 0 : 
g-k^2-kn \ge 0\}.
\end{align*}
Applying the above theorem from $k=0$ to $k=N$, and noting 
that 
\begin{align*}
\pP_0=P_n(S, [H]) \cong \cC^{[n+g-1]}, \ 
\pP_{N+1}=\emptyset
\end{align*}
where the latter is due to (\ref{dim:Mk}), 
we have the following result: 
\begin{cor}\label{cor:KY}
For $n\ge 0$, we have the SOD 
\begin{align}\label{SOD:KKV}
D^b(\cC^{[n+g-1]})
=\langle \aA_0, \aA_1, \ldots, \aA_N \rangle
\end{align}
where each $\aA_k$ has the SOD
\begin{align*}
\aA_k=\langle \overbrace{D^b(U_k), D^b(U_k), \ldots, D^b(U_k)}^{n+2k}  \rangle. 
\end{align*}
\end{cor}

\begin{rmk}\label{rmk:KKV}
As we mentioned in 
in Subsection~\ref{subsec:intro:KY}, the 
SOD (\ref{SOD:KKV}) 
recovers Kawai-Yoshioka's formula (\ref{intro:euler})
\begin{align}\notag
P_{n, g} = (-1)^{n-1}\sum_{k=0}^{N}
(n+2k)e(U_k). 
\end{align}
In~\cite{KY}, the 
formula (\ref{intro:euler}) is the key 
ingredient to prove   
Katz-Klemm-Vafa (KKV) formula for PT
invariants with irreducible curve classes. 
Together with the identities
\begin{align*}
&e(U_k)=e(\Hilb^{g-k(k+n)}(S)), \\
&\sum_{k\ge 0} e(\Hilb^k(S))=
\prod_{k\ge 1}(1-q^k)^{-24}
\end{align*}
the formula (\ref{intro:euler}) is shown 
to imply the following in~\cite{KY}
\begin{align}\label{intro:KKV}
\sum_{g\ge 0} \sum_{n \in \mathbb{Z}}
P_{n, g}z^n q^{g-1}=\left(\sqrt{z}-\frac{1}{\sqrt{z}}  \right)^{-2}
\frac{1}{\Delta(z, q)}.
\end{align}
Here $\Delta(z, q)$ is defined by
\begin{align*}
\Delta(z, q) \cneq q \prod_{n\ge 1}
(1-q^n)^{20}(1-zq^n)^2 (1-z^{-1}q^n)^2. 
\end{align*}
The formula (\ref{intro:KKV})
is the KKV formula mentioned above. 
\end{rmk}

\subsection{Tilting on $S \times C$}
Let $S$ be a K3 surface as in the previous subsection. 
We fix a smooth elliptic curve $C$ and consider 
a compact CY 3-fold $X\cneq S \times C$
with projections $p_S$, $p_C$ 
\begin{align*}
\xymatrix{
X = S \times C \ar[d]_-{p_S} \ar[r]^-{p_C} & C \\
S.
}
\end{align*}
In what follows, we will 
interpret the diagram (\ref{dia:P+}) in terms of wall-crossing 
diagrams in 
$D^b(X)$. 

We define the triangulated subcategory
\begin{align*}
\dD_0 \subset D^b(X)
\end{align*}
consisting of objects 
whose cohomologies are 
supported on fibers of $p_C$. 
The triangulated category $\dD_0$ is the 
derived category of the abelian subcategory
\begin{align*}
\Coh_0(X) \subset \Coh(X)
\end{align*}
consisting of sheaves supported on fibers of $p_C$.
For $c \in C$, let $i_c$ be the inclusion 
\begin{align}\label{ic:include}
i_{c} \colon S \times \{c\} \hookrightarrow S \times C
=X. 
\end{align}
The category $\Coh_0(X)$ is the extension closure of objects 
of the form $i_{c\ast}F$ for some $c \in C$
and 
$F \in \Coh(S)$. 

For $F \in \dD_0$, we
set $v(F) \in \Gamma_S$ to be
\begin{align*}
v(F) \cneq v(p_{S\ast}F)=(v_0(F), v_1(F), v_2(F))
\end{align*}
for $v_i(F) \in H^{2i}(S, \mathbb{Z})$. 
We define the following slope function on 
$\Coh_0(X)$
\begin{align*}
\mu(F) \cneq \frac{v_1(F) \cdot H}{v_0(F)} \in 
\mathbb{Q} \cup \{\infty\}.
\end{align*}
Here $F \in \Coh_0(X)$ and 
we set $\mu(F)=\infty$ if 
$v_0(F)=0$. 
The above slope function on $\Coh_0(X)$ defines the 
$\mu$-stability on it in the usual way: 
an object $E \in \Coh_0(X)$ is defined to be 
\textit{$\mu$-(semi)stable} if for any 
non-zero subsheaf $F' \subsetneq F$, we have 
\begin{align*}
\mu(F') <(\le) \mu(F/F').
\end{align*}
Let $\tT, \fF$ be the subcategories of $\Coh_0(X)$
defined by
\begin{align*}
\tT &\cneq \langle F \in \Coh_0(X) : 
F \mbox{ is } \mu\mbox{-semistable with }\mu(F)>0\rangle_{\rm{ex}}, \\
\fF &\cneq \langle F \in \Coh_0(X) : 
F \mbox{ is } \mu\mbox{-semistable with }\mu(F) \le 0\rangle_{\rm{ex}}.
\end{align*}
Here $\langle - \rangle_{\rm{ex}}$ means the extension closure. 
The pair of subcategories $(\tT, \fF)$ is a torsion 
pair on $\Coh_0(X)$. We have the associated tilting
\begin{align*}
\bB \cneq \langle \fF, \tT[-1] \rangle_{\rm{ex}} \subset \dD_0.
\end{align*}
For $t \in \mathbb{R}_{>0}$, let 
\begin{align}\label{def:Zt}
Z_t \colon K(\dD_0) \to \mathbb{C}
\end{align}
be the group homomorphism 
defined by
\begin{align*}
Z_t(F) &\cneq \int_S e^{-tH \sqrt{-1}} v(F) \\
&= v_2(F)+(1-g)t^2 v_0(F) -(tH \cdot 
v_1(F)) \sqrt{-1}.
\end{align*}
Then the pair 
\begin{align*}
(Z_t, \bB), \ t \in \mathbb{R}_{>0}
\end{align*}
is a Bridgeland stability condition on $\dD_0$
(see~\cite[Lemma~3.3]{TodK3}). 
In particular it defines 
\textit{$Z_t$-(semi)stable objects}: 
an object $E \in \bB$ is 
$Z_t$-(semi)stable if for any 
non-zero subobject $0\neq E' \subsetneq E$
in $\bB$, we have 
the inequality in $(0, \pi]$: 
\begin{align*}
\arg Z_t(E') <(\le) \arg Z_t(E). 
\end{align*}
We have the following lemma: 
\begin{lem}\label{lem:tilting}
An object $E \in \bB$ with 
$v_1(E)=-[H]$, 
$v_0(E) \le 0$ is 
$Z_t$-stable if and only if 
$E\cong i_{c\ast}F[-1]$
for some $c \in C$ and 
$H$-Gieseker stable sheaf $F \in \Coh(S)$.
\end{lem}
\begin{proof}
The lemma is well-known (for example see 
the argument of~\cite[Lemma~6.1]{BayBrill}). 
Let $\cC \subset \bB$ be the subcategory defined by
\begin{align*}
\cC &\cneq \{ F \in \bB : \Imm Z_t(F)=0\} \\
&=\langle U, \oO_x[-1] : 
U \in \Coh_0(X) \mbox{ is } \mu\mbox{-stable with }\mu(U)=0, \ 
x \in X \rangle_{\rm{ex}}.
\end{align*}
Suppose that $E \in \bB$ satisfies 
$v_1(E)=-[H]$
and $v_0(E)\le 0$. 
Since $-v_1(-) \cdot H$
is non-negative on $\bB$, and
 $-v_1(E) \cdot H =H^2$ is the
smallest positive value of $-v_1(-) \cdot H$ on $\bB$, 
the object $E$ 
is $Z_t$-stable if and only if 
$\Hom(\cC, E)=0$. 

First suppose that 
$E$ is $Z_t$-stable, 
so we have $\Hom(\cC, E)=0$.
Then we have 
$\hH^0(E)=0$, and 
 $\hH^1(E)$ is either a $\mu$-stable two 
dimensional sheaf 
or a one dimensional $H$-Gieseker stable sheaf. 
It follows that $E \cong i_{c\ast}F[-1]$
for some $c \in C$, where $F$ is a $\mu$-stable sheaf 
on $S$ or a $H$-Gieseker stable one dimensional sheaf on $S$. 
In the former case, since the Mukai vector of $F$ is primitive, 
the $\mu$-stability of it is equivalent 
to its $H$-Gieseker stability. 
Conversely if $E \cong i_{c\ast}F[-1]$
as in the statement, then 
it is obvious that $\Hom(\cC, E)=0$. 
Therefore the lemma is proved. 
\end{proof}

We define the following subcategory
of $D^b(X)$:
\begin{align}\label{def:aA}
\aA \cneq \langle p_C^{\ast}\Pic(C), \bB \rangle_{\rm{ex}}. 
\end{align}
The category $\aA$ is the heart of a bounded t-structure on the 
triangulated subcategory of 
$D^b(X)$ generated by 
$p_C^{\ast}\Pic(C)$ and objects in $\dD_0$
 (see~\cite[Proposition~2.9]{TodK3}). 
In particular, $\aA$ is 
an abelian category. 
Note that $E \in \aA$ satisfies 
$\rank(E)=0$ if and only if $E \in \bB$. 
We will use the following property on the 
abelian category $\aA$: 
\begin{lem}\label{A:exact}
For any object $E \in \aA$, there is an exact sequence
in $\aA$
\begin{align}\label{A:EEE}
0 \to E' \to E \to E'' \to 0
\end{align}
such that $E' \in \bB$ and 
$E'' \in \langle p_C^{\ast}\Pic(C) \rangle_{\rm{ex}}$. 
\end{lem}
\begin{proof}
The lemma is proved in~\cite[Lemma~7.5]{TodK3}
in the case of $S \times \mathbb{P}^1$, and 
the same argument proves the lemma. 
For simplicity, we prove the lemma 
when $E$ fits into a non-split extension
in $\aA$
\begin{align}\label{E:LF}
0 \to 
p_C^{\ast} \lL \to E \to i_{c\ast}F[-1]
\to 0
\end{align}
for $\lL \in \Pic(C)$, $[F] \in U_k$, 
and $c \in C$. 
The full details are left
to~\cite[Lemma~7.5]{TodK3}. 

Let $\xi$ be the extension class of (\ref{E:LF}). Then 
since $i_c^{!}p_C^{\ast}\lL=\oO_S[-1]$, we have 
\begin{align*}
\xi \in \Ext^2_X(i_{c\ast}F, p_C^{\ast}\lL)
=\Ext_S^1(F, \oO_S). 
\end{align*}
Therefore $\xi$ gives rise to the non-trivial 
extension of sheaves on $S$
\begin{align*}
0 \to \oO_S \to F' \to F \to 0.
\end{align*}
It is easy to see that $F'$ is 
$H$-Gieseker stable so $[F'] \in U_{k+1}$. 
We have the following commutative diagram
\begin{align*}
\xymatrix{
&  E \ar[d] & p_C^{\ast}\lL(c) \ar[d] \\
i_{c\ast}F'[-1] \ar[r] & i_{c\ast}F[-1] \ar[r] \ar[d]^-{\xi} 
& i_{c\ast}\oO_S \ar[d] \\
& p_C^{\ast}\lL[1] \ar[r]^-{\id} & p_C^{\ast}\lL[1]. 
}
\end{align*}
Here horizontal and vertical sequences are distinguished 
triangles. 
By the above diagram, we obtain the exact sequence
in $\aA$
\begin{align*}
0 \to 
i_{c\ast}F'[-1] \to E \to p_C^{\ast}\lL(c) \to 0.
\end{align*}
The above exact sequence is 
a desired sequence (\ref{A:EEE}). 
\end{proof}

\subsection{Weak stability conditions on $\aA$}
Let $\aA$ be the abelian category 
given in (\ref{def:aA}). 
For $t \in \mathbb{R}_{>0}$
and $E \in \aA$, we define
\begin{align*}
\mu_t^{\star}(E) \in \mathbb{R} \cup \{\infty\}
\end{align*}
by the following: 
\begin{align*}
\mu_t^{\star}(E) \cneq \left\{ \begin{array}{cc}
0, & \rank(E)\neq 0, \\
-\frac{\Re Z_t(E)}{\Im Z_t(E)}, & 
\rank(E)=0. 
\end{array} \right.
\end{align*}
Here 
if $\rank(E)=0$, then $E \in \bB$ and 
$Z_t(E) \in \mathbb{C}$ is 
given in (\ref{def:Zt}). 
The following stability condition on $\aA$ 
appeared in~\cite{TodK3} in 
the framework of weak stability conditions: 
\begin{defi}
An object $E \in \aA$ is $\mu_t^{\star}$-(semi)stable 
if for any exact sequence $0 \to E' \to E \to E'' \to 0$
in $\aA$, we have
\begin{align*}
\mu_t^{\star}(E')<(\le) \mu_t^{\star}(E'').
\end{align*}
\end{defi}
Below 
we fix $n \in \mathbb{Z}$ and 
characterize 
$\mu_t^{\star}$-semistable 
objects $E \in \aA$ satisfying
the condition
\begin{align}\label{cond:ch}
\ch(E) &=(1, 0, -i_{c\ast}[H], -n) \\
\notag
& \in H^0(X) \oplus H^2(X) \oplus 
H^4(X) \oplus H^6(X). 
\end{align}

\begin{prop}\label{prop:wall}
For $k \in \mathbb{Z}_{>0}$, suppose that 
$t \in \mathbb{R}_{>0}$ satisfies
\begin{align}\label{t:ineq}
t_k < t < t_{k-1}, \ t_k \cneq \sqrt{\frac{n+k}{(g-1)k}}, \ 
t_0 \cneq \infty. 
\end{align}
Then an object $E \in \aA$ 
satisfying (\ref{cond:ch})
is $\mu_t^{\star}$-semistable 
if and only if 
$E$ is isomorphic to a two term complex
\begin{align}\label{twoterm}
E \cong (p_C^{\ast}\lL \stackrel{s}{\to} i_{c\ast}F)
\end{align}
for some $c \in C$, 
$[F] \in U_k$, 
$\lL \in \Pic^k(C)$ and 
$s$ is a non-zero morphism. 
Here 
$\Pic^k(C) \subset \Pic(C)$
is the subset of degree $k$ line bundles, and 
$p_C^{\ast}\lL$ is located in degree zero.
Moreover in this case, $E$ is 
$\mu_t^{\star}$-stable. 
\end{prop}
\begin{proof}
\begin{sstep}
The `only if' direction. 
\end{sstep}
Let us take 
$t \in (t_{k}, t_{k-1})$ and 
a $\mu_t^{\star}$-semistable
object $E \in \aA$ satisfying (\ref{cond:ch}). 
By~\cite[Lemma~7.5]{TodK3}, 
there is an exact sequence in $\aA$
\begin{align}\label{exact:aec}
0 \to A \to E \to p_C^{\ast}\lL \to 0
\end{align}
for some $A \in \bB$, 
$\lL \in \Pic^r(C)$
for some $r \in \mathbb{Z}$.  
The condition (\ref{cond:ch}) 
implies that 
$v(A)=-v_r$. 
The above exact sequence and the 
$\mu_t^{\star}$-semistability of 
$E$ implies that
\begin{align*}
\Re Z_t(A)=-r-n+r(g-1)t^2 \ge 0.
\end{align*}
As $t<t_{k-1}$, 
the above inequality implies that $r\ge k>0$. 

The $\mu_t^{\star}$-stability of 
$E$ 
implies that $\Hom(\cC, E)=0$, 
where $\cC \subset \bB$ is defined in the proof of Lemma~\ref{lem:tilting}. 
By the exact sequence (\ref{exact:aec}) we have $\Hom(\cC, A)=0$, and 
Lemma~\ref{lem:tilting} shows that 
$A \cong i_{c\ast}F[-1]$
for some $c \in C$ and $[F] \in U_r$.
Therefore $E$ is isomorphic to 
a two term complex 
\begin{align*}
E=(p^{\ast}\lL \stackrel{s'}{\to} 
i_{c\ast}F)
\end{align*}
where $s'$ must be non-zero due to the 
$\mu_t^{\star}$-semistability of $E$. 
Let us show that 
$r=k$. 
By taking the cohomologies of $E$, 
we obtain the exact sequence in $\aA$
\begin{align}\label{exact:rEG}
0 \to p_C^{\ast}\lL(-c) \to E \to G[-1] \to 0
\end{align}
where $G$ is the cokernel of $s'$. 
Since $v(G)=v_{r-1}$, the $\mu_t^{\star}$-semistability 
of $E$
implies that
\begin{align}\label{eqn:ZG}
\Re Z_t(G[-1])=-r+1-n+(r-1)(g-1)t^2 \le 0.
\end{align}
As $t>t_k$, 
the above inequality implies that  
$k\ge r$. As we already proved $r\ge k$, 
it follows that $r=k$. 
Therefore we have proved the only if direction of the proposition. 

\begin{sstep}
The `if' direction.
\end{sstep}
Conversely, let us take an object $E \in \aA$
of the form (\ref{twoterm}). 
We show that $E$ is $\mu_t^{\star}$-stable
if $t \in (t_k, t_{k-1})$. 
Let us take an exact sequence in $\aA$
\begin{align}\label{exact:AEB}
0 \to A \to E \to B \to 0
\end{align}
such that $A$, $B$ are non-zero. 
We will show the inequality
\begin{align}\label{ineq:AB}
\mu_t^{\star}(A)<\mu_t^{\star}(B).
\end{align}
Since $\rank(E)=1$, we have either 
$(\rank(A), \rank(B))=(0, 1)$ or $(1, 0)$. 
We will show (\ref{ineq:AB})
in each cases. 

First suppose that 
$\rank(A)=0$, i.e. $A \in \bB$. 
By the exact sequence in $\aA$
\begin{align}\label{exact:FEk}
0 \to i_{c\ast}F[-1] \to E \to p_C^{\ast}\lL \to 0
\end{align}
we have
$\Hom(\Coh_0(X), E)=0$. 
Therefore we have $\hH^0(A)=0$
and $A \in \tT[-1]$ holds. 
Then $A$ is given as an iterative extensions
of objects of the form $i_{c\ast}T[-1]$ 
for some $c \in C$, 
where $T \in \Coh(S)$ is either torsion or 
$\mu$-stable with $\mu(T)>0$. 
By the Serre duality and the $\mu$-stability of $T$, we have 
\begin{align*}
\Hom(i_{c\ast}T[-1], p^{\ast}_C\lL) &=
\Hom(T, \oO_{S}) \\
&=0. 
\end{align*}
Therefore we have $\Hom(A, p^{\ast}_C\lL)=0$. 
By the exact sequences (\ref{exact:AEB}) and (\ref{exact:FEk}), 
we have the injection 
$A \hookrightarrow i_{c\ast}F[-1]$ in $\bB$. By Lemma~\ref{lem:tilting}, 
the object $i_{c\ast}F[-1]$ is $Z_t$-stable
in $\bB$. Therefore we have
\begin{align*}
\mu_t^{\star}(A) \le \mu_t^{\star}(i_{c\ast}F[-1])
=\mu_t^{\star}(-\mathbf{v}_k)<0=\mu_t^{\star}(B)
\end{align*}
where $\mu_t^{\star}(-\mathbf{v}_k)<0$ is due to $t_k<t$. 
Therefore (\ref{ineq:AB}) holds. 

Next suppose that $\rank(A)=1$, i.e. $B \in \bB$. 
Let $T \subset \hH^0(B)$ be
the HN factor of $\hH^0(B)$ in $\mu$-stability 
such that
$\mu(-)$ is the maximal. 
Note that $\mu(T) \le 0$ 
by the definition of $\bB$. 
If $\mu(T)=0$, then 
$\mu_t^{\star}(T)=\infty$ and we have 
\begin{align*}
\mu_t^{\star}(B) \ge \mu_t^{\star}(B/T).
\end{align*}
Therefore by replacing $B$ by $B/T$, we may assume that 
$\mu(T)<0$. 
This implies the vanishing 
\begin{align}\label{B:vanish}
\Hom(p_C^{\ast}\Pic(C), B)=0.
\end{align}
Similarly to (\ref{exact:rEG}), we have the 
exact sequence in $\aA$
\begin{align}\label{exact:kEG}
0 \to p_C^{\ast}\lL(-c) \to E \to G[-1] \to 0
\end{align}
where $G$ is the cokernel of $s$ in (\ref{twoterm}). 
By the exact sequences (\ref{exact:AEB}), (\ref{exact:kEG}),
and the vanishing (\ref{B:vanish}), 
we see that there is a 
surjection $G[-1] \twoheadrightarrow B$
in $\bB$. By Lemma~\ref{lem:tilting}, 
the object $G[-1] \in \bB$ is $Z_t$-stable.
Therefore we have
\begin{align*}
\mu_t^{\star}(B) \ge \mu_t^{\star}(G[-1])
=\mu_t^{\star}(-\mathbf{v}_{k-1})>0=\mu_t^{\star}(A)
\end{align*} 
where $\mu_t^{\star}(-\mathbf{v}_{k-1})>0$ is due to $t<t_{k-1}$.
Therefore (\ref{ineq:AB}) holds. 
\end{proof}
When $t$ lies on a wall, 
the $\mu_t^{\star}$-semistable objects are 
characterized by the following lemma. 
\begin{lem}\label{lem:wall}
An object $E \in \aA$ 
satisfying (\ref{cond:ch}) is $\mu_{t_k}^{\star}$-semistable 
if and only if 
$E$ is 
$S$-equivalent to a $\mu_{t_k}^{\star}$-polystable 
object of the form 
\begin{align}\label{E:poly}
E_1 \oplus E_2, \ 
E_1=p_C^{\ast}\lL, \ 
E_2=i_{c\ast}F[-1]
\end{align}
for some $c \in C$, $[F] \in U_k$
and $\lL \in \Pic^k(C)$. 
\end{lem}
\begin{proof}
The `if' direction is obvious as 
both of $E_1$, $E_2$ are $\mu_{t_k}^{\star}$-semistable
with $\mu_{t_k}^{\star}(E_1)=\mu_{t_k}^{\star}(E_2)=0$. 
The `only if' direction 
is proved similarly to 
Step~1 in the proof of Proposition~\ref{prop:wall}. 
If we apply the proof above
for $t=t_k$, the only
point to notice is that, 
just after the equation (\ref{eqn:ZG})
we only have $k\ge r-1$
as we take $t=t_k$. 
Therefore we have either $r=k$ or $r=k+1$. 
In the latter case, the exact sequence 
(\ref{exact:rEG})
shows that $E$ is $S$-equivalent to the 
object of the form (\ref{E:poly}). 
\end{proof}

\subsection{Moduli stacks of semistable objects}
Let $\mM$ be 
the 2-functor
\begin{align*}
\mM \colon \sS ch/\mathbb{C} \to \gG roupoid
\end{align*}
sending a $\mathbb{C}$-scheme $S$ to 
the groupoid of relatively perfect objects 
$\eE \in D^b(X \times S)$ such that 
for each point $s \in S$, 
the object $\eE_s \cneq \dL i_s^{\ast}\eE$
for the inclusion 
$i_s \colon X \times \{s\} \hookrightarrow X \times S$
satisfies 
$\Ext^{<0}(\eE_s, \eE_s)=0$. 
The stack $\mM$ is known to be an 
Artin stack locally of finite type~\cite{LIE}. 
For a fixed $n \in \mathbb{Z}_{\ge 0}$
and $t \in \mathbb{R}_{>0}$, 
we consider the substack 
\begin{align}\label{t:substack}
\mM_t^{\star} \subset \mM
\end{align}
to be 
the stack 
whose $S$-valued points
consist
of $\eE \in \mM(S)$
such that 
for each $s \in S$, the object 
$\eE_s$ is a
$\mu_t^{\star}$-semistable object in 
$\aA$ satisfying (\ref{cond:ch}). 
Using Proposition~\ref{prop:wall} and Lemma~\ref{lem:wall}, 
we show the following: 
\begin{prop}\label{prop:stack}
The stack $\mM_t^{\star}$ is an Artin stack of finite type 
such that (\ref{t:substack}) is an open immersion. 
Moreover if $t\in (t_k, t_{k-1})$, the stack 
$\mM_t^{\star}$ is smooth. 
\end{prop}
\begin{proof}
By~\cite[Lemma~4.13]{TodK3} (ii), 
the substack $\mM_t^{\star} \subset \mM$
is constructible. Therefore 
for the first statement, 
it is enough to show that $\mM_t^{\star} \subset \mM$ is open in 
analytic topology. 

By Lemma~\ref{lem:wall}, 
for $t=t_k$ 
an object corresponding to 
a $\mathbb{C}$-valued point of 
$\mM_t^{\star}$ is a small deformation of an
object of the form (\ref{E:poly}). 
Let us set $V^+$, $V^-$ and $U$ 
by
\begin{align*}
&V^+=\Ext_X^1(E_1, E_2), \ V^-=\Ext_X^1(E_2, E_1), \\ 
&U=\Ext_X^1(E_1, E_1) \oplus \Ext_X^1(E_2, E_2).
\end{align*}
Then the analytic local deformation 
space of $E_1 \oplus E_2$ is given by the critical locus of 
some analytic function $w$ 
defined in an analytic neighborhood of $0 \in
V^+ \times V^- \times U$. 
Similarly to the case of stable pairs in (\ref{formal:w}), 
the function 
$w$ is
invariant under the
conjugate $\Aut(E_1 \oplus E_2)=(\mathbb{C}^{\ast})^2$-action 
on $V^+ \times V^- \times U$, 
so it is
 of the form
\begin{align*}
w=w^{(0)}(\vec{u})+\sum_{i, j}x_i y_j w_{ij}^{(1)}(\vec{u})
+\sum_{i, i', j, j'}x_i x_{i'} y_j y_{j'} w_{i i' j j'}^{(2)}(\vec{u}) +\cdots
\end{align*}
where $\vec{x}$, $\vec{y}$ and $\vec{u}$ are 
coordinates of $V^+$, $V^-$ and $U$
respectively. 
As in~\cite[Subsection~5.1]{Toddbir}, 
the function $w$ is constructed using 
the minimal $A_{\infty}$-structure on $D^b(X)$. 
By the construction in \textit{loc. cit. }, the function 
$w^{(0)}(\vec{u})$ is written as 
\begin{align*}
w^{(0)}(\vec{u})=
w^{(0)}_1(\vec{u}_1)+w^{(0)}_2(\vec{u}_2), \ 
\vec{u}=(\vec{u}_1, \vec{u}_2), \ 
\vec{u}_i  \in \Ext_X^1(E_i, E_i)
\end{align*}
such that the critical locus of 
$w^{(0)}_i(\vec{u}_i)$ 
in $\Ext_X^1(E_i, E_i)$
gives the 
local deformation space of $E_i$. 
Since the deformation space of $E_i$ 
is smooth, we may assume that $w^{(0)}(\vec{u})=0$. 

Similarly to Subsection~\ref{subsec:wallpair}, 
the function 
$w_{ij}^{(1)}(\vec{u})$ is written 
as (\ref{wij(1)})
such that the coefficients of the linear term
$a_{ijk}$ is determined by the 
triple product
\begin{align*}
V^+ \times V^- \times U  \to \mathbb{C}
\end{align*}
given by the composition and the Serre duality. 
Then by Lemma~\ref{lem:inj2} below, 
the coefficients $a_{ijk}$ satisfy the 
condition in Assumption~\ref{assum:cond} (ii). 
Therefore  
the argument of Lemma~\ref{lem:smooth} shows that 
\begin{align}\label{small:deform}
&(dw=0) \cap  (V^{+\ast} \times V^- \times U)
\subset (V^{+\ast} \times \{0\} \times U), \\
\notag
&(dw=0) \cap  (V^+ \times V^{-\ast} \times U)
\subset (\{0\} \times V^{-\ast} \times U).
\end{align}
This implies that any small deformation $E'$ of 
$E_1 \oplus E_2$ fits into one of the following exact sequences
in $\aA$
\begin{align}\label{small:deform2}
&0 \to i_{c'\ast}F'[-1] \to E' \to p_C^{\ast} \lL' \to 0, \\ 
\notag
&0 \to p_C^{\ast}\lL' \to E' \to i_{c\ast}F'[-1] \to 0
\end{align}
where $(F', \lL', c')$ is a small deformation of $(F, \lL, c)$, so that 
$[F'] \in U_k$ and $\lL' \in \Pic^k(C)$. 
Therefore $E'$ is $\mu_{t_k}^{\star}$-semistable,
and $\mM_{t_k}^{\star} \subset \mM$ is open 

Suppose that $t \in (t_{k}, t_{k-1})$, and take 
an object $E$ as in (\ref{twoterm})
which corresponds to a $\mathbb{C}$-valued point of 
$\mM_t^{\star}$. 
Then $E$ is isomorphic to a small deformation of the object
$E_1 \oplus E_2$ as above, which lies in the 
LHS of (\ref{small:deform}). 
Then any small deformation $E'$ of $E$ 
fits into a non-split sequence (\ref{small:deform2}). 
Therefore 
$E'$ is again $\mu_t^{\star}$-semistable by 
Proposition~\ref{prop:wall}, 
and $\mM_t^{\star} \subset \mM$ is open. 
Moreover the argument of Lemma~\ref{lem:smooth}
implies that the LHS of (\ref{small:deform}) is smooth, 
hence $\mM_t^{\star}$ is smooth. 
\end{proof}
We have used the following lemma, 
which is an analogy of Lemma~\ref{lem:inj}. 
\begin{lem}\label{lem:inj2}
For the objects $E_1$, $E_2$ in (\ref{E:poly}), the 
composition map 
\begin{align}\label{compose:inj}
\Ext_X^1(E_1, E_2) \otimes \Ext_X^1(E_2, E_1)
 \to \Ext_X^2(E_2, E_2)
\end{align}
is injective. 
\end{lem}
\begin{proof}
Note that we have
\begin{align*}
\Ext_X^1(E_1, E_2)=H^0(S, F), \ 
\Ext_X^1(E_2, E_1)=\Ext_S^1(F, \oO_S).
\end{align*}
We also have the surjection
\begin{align*}
\Ext_X^2(E_2, E_2)=\Ext_X^2(i_{c\ast}F, i_{c\ast}F) 
\twoheadrightarrow \Ext_S^1(F, F)
\end{align*}
which is Serre dual to the natural 
map 
$\Ext_S^1(F, F) \to \Ext_X^1(i_{c\ast}F, i_{c\ast}F)$. 
By composing it with (\ref{lem:inj2}), we obtain the 
composition map
\begin{align}\label{map:HF}
H^0(S, F) \otimes \Ext_S^1(F, \oO_S) \to \Ext_S^1(F, F). 
\end{align}
It is enough to 
show that the map (\ref{map:HF}) is injective.
Let us take the universal extension in 
$\Coh(S)$
\begin{align}\label{universal}
0 \to \Ext_S^1(F, \oO_S)^{\vee} \otimes \oO_S \to \uU 
\to F \to 0.
\end{align}
Then it is well-known that $\uU$ is a $\mu$-stable sheaf 
(see~\cite{MR1717621, noteBG}). 
Applying $\Hom(-, F)$ to the above exact sequence, we obtain 
the exact sequence
\begin{align}\label{exact:SF}
0 \to \mathbb{C} \to \Hom(\uU, F) \to 
H^0(S, F) \otimes \Ext_S^1(F, \oO_S) \to \Ext_S^1(F, F).
\end{align}
Since (\ref{universal})
is the universal extension, 
applying $\Hom(-, \oO_S)$ to (\ref{universal}) we obtain 
$\Ext_S^1(\uU, \oO_S)=0$. 
Then applying $\Hom(\uU, -)$ to (\ref{universal})
and using the stability of 
$\uU$, we obtain 
\begin{align*}
\Hom(\uU, F)=\Hom(\uU, \uU)=\mathbb{C}.
\end{align*} 
Therefore by the exact sequence (\ref{exact:SF}), we 
see that (\ref{map:HF}) is injective. 
\end{proof}

For $t \in \mathbb{R}_{>0}$, let 
\begin{align}\label{good:M}
\mM_t^{\star} \to M_t^{\star}
\end{align}
be the good moduli space for 
the stack $\mM_t^{\star}$, 
which exists by~\cite{AHLH}. 
The good 
moduli space $M_t^{\star}$ is an algebraic 
space of finite type, which parametrizes
$\mu_t^{\star}$-polystable objects in $\aA$
satisfying (\ref{cond:ch}), 
i.e. direct sums of $\mu_t^{\star}$-stable 
objects with $\mu_t^{\star}(-)=0$.
By Proposition~\ref{prop:wall}, 
the moduli space $M_t^{\star}$ is constant if $t \in (t_k, t_{k-1})$
for some $k$.
So we can write
\begin{align*}
M_k^{\star} \cneq M_t^{\star}, \ 
t \in (t_k, t_{k-1}). 
\end{align*} 
By Proposition~\ref{prop:wall}, 
$M_k^{\star}$ consists of 
$\mu_t^{\star}$-stable objects
for $t \in (t_k, t_{k-1})$
and also smooth by 
Proposition~\ref{prop:stack}. 

Recall that $J_C \cneq \pP ic^0(C)$ is defined 
to be the moduli space of degree zero line bundles on $C$, 
which is isomorphic to $C$ itself as $C$ is an elliptic curve. 
In the $k=1$ case, 
we can describe $M_1^{\star}$ 
by the stable pair moduli space: 
\begin{lem}
For $\beta=i_{c\ast}[H]$, we have 
the isomorphism 
\begin{align}
\label{isom:k=0:2}
P_{-n}(X, \beta) \times J_C \stackrel{\cong}{\to}
M_1^{\star}
\end{align}
given by 
\begin{align}
\label{map:k=1}
((\oO_X \to i_{c\ast}F'), L)
&\mapsto 
p_C^{\ast}L \otimes \mathbb{D}(\oO_X \to i_{c\ast}F'). 
 \end{align}
Here $\mathbb{D} \cneq \dR \hH om(-, \oO_X)$
is the derived dual. 
\end{lem}
\begin{proof}
First we need to 
show that 
the map (\ref{map:k=1}) is well-defined, i.e. 
the object 
\begin{align}\label{pcast}
p_C^{\ast}L \otimes \mathbb{D}(\oO_X \to i_{c\ast}F')
\in \aA
\end{align}
in the RHS of (\ref{map:k=1})
corresponds to a point in $M_1^{\star}$. 
By~\cite[Remark~9.8]{Toddbir},
an object in $E\in \aA$ is of the form (\ref{pcast})
if and only if $E$ fits into an exact sequence
in $\aA$
\begin{align}\label{pCLE}
0 \to p_C^{\ast}L \to E \to i_{c\ast}F''[-1] \to 0
\end{align}
where $F''$ is a pure one dimensional 
sheaf on $S$ such that 
$\Hom(T[-1], E)=0$ for any 
one dimensional sheaf $T$ on $X$. 
Moreover in this case, 
we have $i_{c\ast}F''=\eE xt_X^2(i_{c\ast}F', \oO_X)$. 
The proof of Lemma~\ref{A:exact} shows that $E$ fits into 
an exact sequence
\begin{align}\label{pCLE2}
0 \to i_{c\ast}F'''[-1] \to E \to p_C^{\ast}L(c) \to 0
\end{align}
for $[F'''] \in U_1$. 
Therefore $E$ is isomorphic to 
$(p_C^{\ast}L(c) \stackrel{s}{\to} i_{c\ast}F''')$
for a non-zero $s$, 
hence gives a point in $M_1^{\star}$ by 
Proposition~\ref{prop:wall}. 

Conversely by Proposition~\ref{prop:wall}, 
any object $[E] \in M_1^{\star}$
fits into an exact sequence of the form 
(\ref{pCLE2}).
By taking the cohomologies of $E$, 
 it also fits into a non-split
 exact sequence of the form (\ref{pCLE}).
On the other hand,
 by the exact sequence (\ref{pCLE2}) 
we see that
$\Hom(T[-1], E)=0$ for any 
one dimensional sheaf $T$ on $X$.
Therefore $E$ is of the form (\ref{pcast}), 
and the map (\ref{map:k=1})
is bijective on closed points. 
Since both sides of (\ref{isom:k=0:2}) 
are smooth, 
it is an isomorphism.  
\end{proof}

In general for $k>0$, we can 
describe $M_k^{\star}$ in terms of 
pair moduli spaces $\pP_k$ on $S$: 
\begin{lem}\label{lem:isom:star}
For $k>0$, we have an isomorphism
\begin{align}\label{isom:star1}
\pP_k \times (C \times J_C) 
\stackrel{\cong}{\to}
M_k^{\star}
\end{align}
given by 
\begin{align}\label{isom:form}
((\oO_S \to F), c, L) \mapsto 
\left(
p_C^{\ast}(\oO_C(k[c]) \otimes L)
\to i_{c\ast}F\right).
\end{align}
\end{lem}
\begin{proof}
The map 
(\ref{isom:form}) is a morphism of smooth 
algebraic spaces which is bijective on closed points
by Proposition~\ref{prop:wall}.
Hence (\ref{isom:form}) is an isomorphism. 
\end{proof}
We also set
\begin{align*}
U_k^{\star} \cneq (M_{t_k}^{\star})^{\rm{red}}, \ 
k \in \mathbb{Z}_{> 0}. 
\end{align*}
By the open immersions
\begin{align*}
\mM_{t_k+\varepsilon}^{\star} \subset \mM_{t_k}^{\star} \supset
\mM_{t_k-\varepsilon}^{\star}
\end{align*}
for $0<\varepsilon \ll 1$, 
and noting that $M_k^{\star}$ is smooth, 
we have the induced morphisms
\begin{align}\label{dia:Mstar}
\xymatrix{
M_k^{\star} \ar[rd] _-{\pi_k^{\star+}}& & M_{k+1}^{\star} \ar[ld]^-{\pi_k^{\star-}} \\
& U_k^{\star}. &
}
\end{align}

\begin{lem}\label{lem:isom:star2}
(i) We have an isomorphism
\begin{align}\label{isom:star2}
U_k \times (C \times J_C) 
\stackrel{\cong}{\to}
U_k^{\star}.
\end{align}

(ii)
Under the isomorphisms (\ref{isom:star1}),
(\ref{isom:star2}), 
 the diagram (\ref{dia:Mstar}) is identified with the diagram 
$(\ref{dia:P+}) \times \mathrm{id}_{C \times J_C}$.   
\end{lem}
\begin{proof}
(i) By Lemma~\ref{lem:wall}, a point 
in $M_{t_k}^{\star}$ corresponds to a 
$\mu_{t_k}^{\star}$-polystable object of the form
(\ref{E:poly}). 
Therefore we have the morphism
\begin{align*}
U_k \times (C \times J_C) \to M_{t_k}^{\star}
\end{align*}
defined by 
\begin{align}\label{mor:UM}
(F, c, L)
 \mapsto p_C^{\ast}(\oO_C(k[c]) \otimes L) \oplus i_{c\ast}F[-1].
\end{align}
The morphism (\ref{mor:UM}) 
a bijection on closed points. 
Moreover the argument of~\cite[Lemma~9.21]{Toddbir}
shows that 
(\ref{mor:UM}) is a closed immersion. 
Therefore we have the isomorphism (\ref{isom:star2})
by taking the reduced parts of (\ref{mor:UM}). 

(ii) 
The statement $\pi_k^{\star +}=\pi_k \times (\id_{C \times J_C})$ 
is obvious from the descriptions of 
the maps (\ref{isom:form}), (\ref{mor:UM}). 
As for $\pi_k^{\star -}$, 
let us take a point
\begin{align*}
((\oO_S \stackrel{s'}{\to} F'), c, L')
\in \pP_{k+1} \times (C \times J_C).
\end{align*} 
Under the map (\ref{isom:star1}), 
it corresponds to a point in $M_{k+1}^{\star}$ of the form 
\begin{align*}
E'=(p_C^{\ast}\lL'\to i_{c\ast}F') \in \aA, \ 
\lL'=\oO_C((k+1)[c]) \otimes L' \in \Pic^{k+1}(C).
\end{align*}
By taking the cohomologies of $E'$, 
we have an exact sequence in $\aA$
\begin{align*}
0 \to p_C^{\ast}\lL'(-c) \to E' \to G[-1] \to 0
\end{align*}
where $G$ is the cokernel of $s'$. 
Then the map $\pi_k^{\star -}$ is given by
\begin{align*}
\pi_k^{\star -}(E')=
p_C^{\ast}\lL'(-c) \oplus G[-1]. 
\end{align*}
As $\lL'(-c)=\oO_C(k[c]) \otimes L'$, 
it comes from $(G, c, L') \in U_k \times (C \times J_C)$
under the map (\ref{mor:UM}). 
Therefore 
the identity 
$\pi_k^{\star - }=\pi_k^- \times \id_{C \times J_C}$
also 
holds. 
\end{proof}

\begin{prop}\label{prop:MU}
The diagram (\ref{dia:Mstar})
satisfies
Assumption~\ref{assum:cond}
by setting 
\begin{align*}
M^+=M_k^{\star}, \ M^-=M_{k+1}^{\star}, \ 
U=U_k^{\star}, \ 
\pi^{\pm}=\pi_k^{\star\pm}.
\end{align*}
\end{prop}
\begin{proof}
Note that the diagram (\ref{dia:Mstar}) is a wall-crossing diagram 
in the CY 3-fold $X$.
Together with the fact that 
a point in $U_k^{\star}$ corresponds to a $\mu_t^{\star}$-polystable 
object (\ref{E:poly}), it is a
d-critical simple flip by~\cite[Example~6.3]{Toddbir}
(also see the argument of~\cite[Theorem~9.22]{Toddbir}). 
Therefore Assumption~\ref{assum:cond} (i) holds. 
By Lemma~\ref{lem:inj2}
the Assumption~\ref{assum:cond} (ii) holds, 
and Assumption~\ref{assum:cond} (iii) 
holds by the same argument of Lemma~\ref{lem:O1}. 
\end{proof}

\subsection{Proof of Theorem~\ref{thm:K3}}\label{subsec:proofK3}
\begin{proof}
We first prove Theorem~\ref{thm:K3} for $k>0$. 
Let $\wW_k$ be the fiber product of the diagram (\ref{dia:Mstar}), and 
$\oO_{M_k^{\star}}(1)$ be 
a $\pi_k^{\star+}$-ample line bundle on $M_k^{\star}$
satisfying Assumption~\ref{assum:cond} (iii)
for the diagram (\ref{dia:Mstar}). 
By Theorem~\ref{thm:main} and Proposition~\ref{prop:MU}, 
we have the fully-faithful functors
\begin{align}\label{funct:star}
\Phi^{\oO_{\wW_k}}
&\colon D^b(M_{k+1}^{\star}) \hookrightarrow D^b(M_k^{\star}), \\
\notag
\Upsilon^{i}_k &\colon 
D^b(U_k^{\star}) \hookrightarrow D^b(M_k^{\star}).
\end{align}
Here $\Upsilon^{i}_k$ is given by 
$\dL(\pi_k^{\star+})^{\ast}(-) \otimes \oO_{M_k^{\star}}(i)$. 
Moreover we have the SOD
\begin{align}\label{SOD:star}
D^b(M_k^{\star})=\langle \Imm \Upsilon^{-2k-n+1}_k, \ldots, 
\Imm \Upsilon^{0}_k, \Imm 
\Phi^{\oO_{\wW_k}} \rangle.
\end{align}
Then by Lemma~\ref{lem:isom:star2} (ii), 
the functors (\ref{funct:star})
are linear over $C \times J_C$ under the 
isomorphisms (\ref{isom:star1}), (\ref{isom:star2}), 
so Theorem~\ref{thm:K3} for $k>0$
follows by 
restricting the
SOD (\ref{SOD:star})
to $U_k \times \{(0, 0)\} \subset U_k^{\star}$
(see~\cite[Proposition~5.1, Theorem~6.4]{MR2801403}). 

Finally we prove Theorem~\ref{thm:K3}
for $k=0$. 
By setting 
$\beta=i_{c\ast}[H]$ we define
\begin{align*}
M_0^{\star} \cneq 
P_n(X, \beta)\times J_C, \ 
U_0^{\star} \cneq U_n(X, \beta) \times J_C.
\end{align*}
Then we have the diagram 
\begin{align}\label{dia:Mstar2}
\xymatrix{
M_0^{\star} \ar[rd] _-{\pi_0^{\star+}}& & M_{1}^{\star} \ar[ld]^-{\pi_0^{\star-}} \\
& U_0^{\star} &
}
\end{align}
by taking the product of the diagram (\ref{dia:PT})
with $J_C$ via the isomorphism (\ref{isom:k=0:2}). 
The diagram (\ref{dia:Mstar2}) satisfies 
Assumption~\ref{assum:cond}
as in the proof of Theorem~\ref{thm:PTpair}.
On the other hand,  
similarly to Lemma~\ref{lem:isom:star} and Lemma~\ref{lem:isom:star2}, 
we have 
isomorphisms
\begin{align}\notag
&\pP_0 \times (C \times J_C) \stackrel{\cong}{\to}
M_0^{\star}, \ 
((\oO_S \stackrel{s}{\to} F), c, L)
\to ((\oO_X \stackrel{s}{\to}i_{c\ast}F), L), \\
\notag
&U_0 \times (C \times J_C) \stackrel{\cong}{\to}
U_0^{\star}, \ (F, c, L) \mapsto (i_{c\ast}F, L). 
\end{align}
Under the above isomorphisms, 
the diagram (\ref{dia:Mstar2}) is identified
with the diagram $(\ref{dia:P+}) \times \mathrm{id}_{C \times J_C}$
for $k=0$.  
Therefore the argument for $k>0$ also implies
Theorem~\ref{thm:K3} for $k=0$. 
\end{proof}

\bibliographystyle{amsalpha}
\bibliography{math}

Kavli Institute for the Physics and 
Mathematics of the Universe, University of Tokyo (WPI),
5-1-5 Kashiwanoha, Kashiwa, 277-8583, Japan.

\textit{E-mail address}: yukinobu.toda@ipmu.jp

\end{document}